\documentclass[11pt,leqno]{article}
\usepackage{amssymb,amsfonts}
\usepackage{amsmath,amsthm,amsxtra}
\usepackage{epsfig}
\usepackage{slashed}
\usepackage{color}
\usepackage{verbatim}
\usepackage[notcite,notref]{showkeys}

\newcommand\bigcheck[1]{#1 \raise1ex\hbox{$\hspace{-1ex}{}^\vee$}}
\newcommand\sucheck[1]{#1 \raise0.5ex\hbox{$\hspace{-1ex}{}^\vee$}}

\newcommand{\add}{{\rm add}}

\newcommand{\End}{\mathop{\rm End }}

\renewcommand{\Im}{\mathop{\rm Im  \, }}

\renewcommand{\sl}{s\ell}
\newcommand{\str}{{\rm str}}

\renewcommand{\Re}{\mathop{\rm Re  \, }}

\newcommand{\cF}{\mathcal{F}}
\newcommand{\cG}{\mathcal{G}}

\newcommand{\CC}{\mathbb{C}}

\newcommand{\QQ}{\mathbb{Q}}
\newcommand{\RR}{\mathbb{R}}
\newcommand{\ZZ}{\mathbb{Z}}

\newcommand{\fg}{\mathfrak{g}}
\newcommand{\fh}{\mathfrak{h}}

\newcommand\bl{(\, . \, | \, . \, )}

\renewcommand{\tilde}{\widetilde}
\renewcommand{\hat}{\widehat}

\makeatletter
\renewcommand\section{\@startsection {section}{1}{\z@}%
                                   {-3.5ex \@plus -1ex \@minus -.2ex}%
                                   {2.3ex \@plus.2ex}%
                                   {\normalfont\large\bfseries}}
\renewcommand\subsection{\@startsection{subsection}{2}{\z@}%
                                     {-3.25ex\@plus -1ex \@minus -.2ex}%
                                     {0ex \@plus .0ex}%
                                     {\normalfont\normalsize\bfseries}}

\@addtoreset{equation}{section}
\makeatother

\newtheorem{theorem}{Theorem}[section]
\newtheorem{definition}[theorem]{Definition}
\newtheorem{lemma}[theorem]{Lemma}

\newtheorem*{lemma*}{Lemma}

\theoremstyle{remark}
\newtheorem{remark}[theorem]{Remark}
\newtheorem{example}[theorem]{Example}

\makeatletter
\def\@maketitle{\newpage
 \null
 \vskip 2em
 \begin{center}%
 \vskip 3em
  {\Large\bf \@title \par}%
  \vskip 1.5em
  {\normalsize
   \lineskip .5em
   \begin{tabular}[t]{c}\@author
   \end{tabular}\par}%
  \vskip 2em

 \end{center}%
 \par
 \vskip 2.5em}
\makeatother

\renewcommand{\epsilon}{\varepsilon}

\definecolor{light}{gray}{.9}

\newcommand{\half}{\frac{1}{2}}
\newcommand{\thalf}{\tfrac{1}{2}}
\newcommand{\zp}{\ZZ_{>0}}

\newcommand{\la}{\lambda}

\newcommand{\al}{\alpha}
\newcommand{\osp}{\mathrm{osp} \,}
\newcommand{\wg}{\widehat{\fg}}
\newcommand{\wh}{\widehat{\fh}}
\newcommand{\thl}{\Theta_{\lambda}}
\newcommand{\tzt}{(\tau, z, t)}
\newcommand{\tuvt}{(\tau, u, v, t)}

\begin{document}

\title{A characterization of modified mock theta functions}


\author{Victor G.~Kac\thanks{Department of Mathematics, M.I.T, 
Cambridge, MA 02139, USA. Email:  kac@math.mit.edu~~~~Supported in part by an NSF grant.} \ 
and Minoru Wakimoto\thanks{Email: ~~wakimoto@r6.dion.ne.jp~~~~.
Supported in part by Department of Mathematics, M.I.T.}}

\maketitle

\setcounter{section}{-1}

\section*{Abstract}
We give a characterization of modified (in the sense of Zwegers) mock theta functions, parallel to that of ordinary theta functions. Namely, modified mock theta functions are characterized by their analyticity properties, elliptic transformation properties, and by being annihilated by certain second order differential operators.

\section{Introduction}
The mock theta functions (also called Appell's functions \cite{Ap}, \cite{KW2}; or Lerch sums in \cite{Z}) of degree $m$ are defined by the following series:
\begin{equation}
\label{e0.1}
\Phi^{\pm [m,s]} (\tau, z_1, z_2, t) = e^{2\pi i mt} \sum_{n \in \ZZ} (\pm 1)^n \ \frac{q^{mn^2 + ns} e^{2 \pi i (mn(z_1 + z_2) + sz_1)}}{1 - e^{2 \pi i z_1}q^n}, 
\end{equation}
\noindent where $ m \in \zp, s \in \ZZ  $ (resp. $ m \in \half + \ZZ_{\geq 0}, s \in \half + \ZZ $) in case of $ + $ (resp. $ - $). These series converge to meromorphic functions in the domain $ X = \{ (\tau, z_1, z_2, t) \in \CC^4 \ | \ \Im \tau > 0\}. $

These kind of functions first appeared in Appell's study of elliptic functions of the third kind in the 1890's \cite{Ap}. A number of identities for various specializations of these functions have been discovered in the attempts to understand Ramanujan's mock theta functions (hence the name for the functions $ \Phi^{\pm [m,s]} $). On the other hand, it has been understood that the numerators of the normalized supercharacters of the lowest rank affine Lie superalgebras (of positive defect) $ \widehat{s \ell}_{2|1} $  and
 $ \widehat{\osp}_{3|2} $ can be expressed in terms of the functions  
$ \Phi^{\pm[m,s]} $,
see \cite{KW2}--\cite{KW4}. (For the higher rank affine Lie superalgebras one needs higher rank mock theta functions, see \cite{KW1}--\cite{KW5} and \cite{GK}, which are not considered in this paper.)

Consider in more detail the example of the simple Lie superalgebra $ \fg = s \ell_{2|1} $ over $\CC$ with the invariant bilinear form $ (a | b) = \str \ ab$.
Choose its Cartan subalgebra $ \fh, $ consisting of supertraceless diagonal matrices. The associated affine Lie superalgebra is the infinite-dimensional Lie superalgebra over $ \CC: $
\[ \wg = \fg [t, t^{-1}] \oplus \CC K \oplus \CC d,  \]
\noindent with the following commutation relations $ (a, b \in \fg, m,n \in \ZZ): $
\[ [at^m, bt^n] = [a,b]t^{m+n} + m \delta_{m,-n} (a|b)K, \quad [d, at^m] = mat^m, \quad [K, \wg] = 0. \]
\noindent Let $ \wh = \CC d + \fh + \CC K $ be the Cartan subalgebra of $ \wg $ with the following coordinates for $ h \in \wh: $
\[ h = 2 \pi i (-\tau d + z_1 (E_{22} + E_{33}) -z_2 (E_{11} + E_{33}) + tK), \mbox{ where } \tau, z_1, z_2, t \in \CC. \]
\noindent For each $ m > 0 $ the Lie superalgebra $ \wg $ has a unique irreducible module $ V_m,  $ such that $ K = mI_{V_m} $ and there exists a non-zero vector $ v_m \in V_m $ for which $ (\fg [t] + \CC d)v_m = 0. $ We showed in \cite{KW2} that the (normalized)  supercharacter of $ V_1 $ is given by the following formula ($ h \in \wh $):
\[ \str_{V_1} e^h = e^{2 \pi i t} \eta (\tau)^3 \vartheta_{11} (\tau, z_1) \vartheta_{11} (\tau, z_2) \mu (\tau, z_1, z_2), \]
\noindent where
\[ \mu (\tau, z_1, z_2) = \vartheta_{11} (\tau, z_2)^{-1} \Phi^{-[\half, \half]} (\tau, z_1, 2z_2-z_1). \]
\noindent The function $ \mu (\tau, z_1, z_2) $ is the prototype for a mock theta function in the sense that specializing the complex variables $ z_1 $ and $ z_2 $ to torsion points (i.e. elements of $ \QQ + \QQ \tau $), one gets mock $ \vartheta $-functions in the sense of Ramanujan, see \cite{Z}, \cite{Za}.

An important discovery of Zwegers is the real analytic, but not meromorphic, function $R (\tau,u)$, $\tau, u \in \CC$, $\Im \tau >0$, such that the modified function
\[ \tilde{\mu} (\tau,z_1,z_2) = \mu (\tau,z_1,z_2) + \frac{i}{2} R (\tau,z_1-z_2) \]
\noindent is a modular invariant function with nice elliptic transformation properties (\cite{Z}, Theorem~1.11).  Furthermore, Zwegers introduces real analytic functions $R_{m,\ell} (\tau,u)$, similar to 
$R (\tau ,u)$ 
(in fact, $R (\tau ,u)=R_{2,1} (\tau ,u/2)-R_{2,-1}(\tau ,u/2)$), 
 such that, adding to a rank $1$ mock theta function of degree $m$  a suitable linear combination of rank $1$ Jacobi forms $\Theta_{m,\ell}$ as coefficients, he obtains a modular invariant real analytic function \cite{Z}, Proposition~3.5 (see (\ref{e2.10}), (\ref{e2.11}) for our version of this construction).  The latter functions are used in the study of Ramanujan's mock theta functions (\cite{Z}, Chapter~4).

In our paper \cite{KW3}, Section 5, we used the functions $R_{m+1,\ell}$ of Zwegers (see (\ref{e2.09}) for our version of these functions) in order to modify the (normalized) supercharacter of the $\hat{\sl}_{2|1}$-module $V_m$, where $m$ is a positive integer.  The normalized supercharacter is given in this case by the following formula:
\[      \widehat{R}^A (\tau, z_1, z_2, t) \ \str_{\, V_m} e^h  = \Phi^{+[m+1,0]}(\tau,z_1,z_2,t)-\Phi^{+[m+1,0]} (\tau, -z_2, -z_1, t), \]
\noindent where $\widehat{R}^A$ is the affine superdenominator (see (\ref{e2.21})  for its expression).

Following Zwegers' ideas, we introduced in \cite{KW3} the real analytic modified numerator
\[ \tilde{\Phi}^{+[m+1,0]} (\tau, z_1, z_2, t) - \tilde{\Phi}^{+[m+1, 0]} (\tau, -z_2, -z_1), \]
\noindent where 
\[ \tilde{\Phi}^{+[m+1,0]} (\tau,z_1,z_2,t) = \Phi^{+[m+1, 0]}  (\tau,z_1,z_2,t) - \half  \Phi^{+[m+1,0]}_{\add}  (\tau,z_1,z_2,t), \]
and $\Phi^{+[m+1,0]}_{\add}$ is a real analytic function (defined by (\ref{e2.09})--(\ref{e2.11}) of the present paper), and proved its modular and elliptic transformation properties. This establishes modular invariance of the modified (normalized) supercharacter  
 \[ \left(\tilde{\Phi}^{+[m+1,0]}(\tau, z_1, z_2, t) - \tilde{\Phi}^{+[m+1,0]} (\tau, -z_2, -z_1, t) \right) / \widehat{R}^A (\tau, z_1, z_2, t)  \]
\noindent  of $ V_m. $

In \cite{KW4} and \cite{KW5} we study, in a similar fashion, the supercharacters of more general integrable $ \widehat{\sl}_{2|1} $ (resp. $ \widehat{\osp}_{3|2} $)-modules, which use the functions $ \Phi^{+[m,s]} $ (resp. $ \Phi^{-[m,s]} $).

In the earlier paper \cite{KP} close connections of character theory of integrable modules over affine Lie algebras $ \wg  $ and the theory of theta functions have been displayed. The key role in this theory was played by the fact that the translations subgroup (of finite index) of the Weyl group of $ \wg $ consists of transformations $ t_{\al} \in \End \wh, $ given by formula (\ref{e1.06}), where $ \al $ runs over the coroot lattice of the simple Lie algebra $ \fg. $  This has lead to a simple characterization of theta functions (recalled in Definition \ref{d1.02}).

In the present paper we propose a similar characterization of the modified mock theta functions 
\[ \tilde{\varphi}^{\pm[m,s]} \tuvt := \tilde{\Phi}^{\pm[m,s]} (\tau, v-u, -v-u, t). \]

Note that, like in the theta function case, one can construct higher rank mock theta functions (see (\ref{e2.03}), (\ref{e2.04})). Moreover, in \cite{KW5} we construct an inductive modification of these functions and use them to construct modular invariant families of modified normalized supercharacters for affine Lie superalgebras $ \wg, $ where $ \fg $ is a basic simple Lie superalgebra, different from $ p s \ell_{n|n}. $ However, we do not know of any characterization of these higher rank modified mock theta functions, similar to that of theta functions. 

The contents of this paper are as follows. In Section 1 we recall the axiomatic definition of the spaces of degree $ m $ theta function $\mathrm{Th}^\pm_m$ (see Definitions \ref{d1.02} and \ref{d1.04}) and their modular transformation properties, following \cite{KP} and \cite{K2} (in the case of $ + $).We discuss in detail the rank 1 theta functions (= Jacobi forms) $ \Theta_{j,m} $ and, in particular, the four Jacobi forms $ \vartheta_{ab}, \ a,b = 0 $ or 1. 

In Section 2 we introduce degree $ m $ mock theta functions $ \Theta^\pm_{\la; B} $ of arbitrary rank (which are theta functions if $ B =\emptyset $), following [KW3--KW5]. Following \cite{Z} and [KW3--KW5] we construct the modification $ \tilde{\Phi}^{\pm[m,s]} $ of rank 1 mock theta functions $ \Phi^{\pm[m,s]} $. Furthermore, we introduce axiomatically the spaces of functions $ \cF^{[m; s, s']}, $ where $ m \in \half \zp, \ s, s' \in \half \ZZ,  $ and state our main Theorems \ref{t2.2} and \ref{t2.3}. Theorem \ref{t2.2} states that for $ m>1 $ and for $ m = 1, (s,s') \notin \ZZ^2, $ the space $ \cF^{[m;s,s']} $ is spanned (over $ \CC  $) by the modified mock theta functions $ \tilde{\varphi}^{+[m,s]} $ (resp. $ \tilde{\varphi}^{-[m,s]} $) if $ s' \in \ZZ $ (resp. $ \in \half + \ZZ $). Theorem \ref{t2.3} covers the remaining cases; in particular it states that in these cases $ \cF^{[m;s,s']} $ is the span of $ \tilde{\varphi}^{+[m,s]} $ (resp. $ \tilde{\varphi}^{-[m,s]} $) if $ s' \in \ZZ $ (resp. $ \in \half + \ZZ $) and a holomorphic affine superdenominator.

In Section 3 we give proofs of the main theorems, based on Lemmas \ref{L3.1}--\ref{L3.3}, and in Section 4 we prove these lemmas. 

\section{A brief theory of theta functions}
Let $ \fh $ be an $ \ell $-dimensional vector space over $ \CC, $ endowed with a non-degenerate symmetric bilinear form $ \bl. $ Let $ m $ be a positive real number and let $ L $ be a lattice (i.e. a free abelian subgroup) in $ \fh,$ such that 
\begin{equation}
\label{e1.01}
m (\al | \beta) \in \ZZ \mbox{ for all } \al, \beta \in L.
\end{equation}
\noindent This condition means that $ mL \subset L^*, $ where $ L^* = \{ \la \in \CC L \ | \ (\la | L) \subset \ZZ   \} $ is the dual lattice. 

Let $ \widehat{\fh} = \fh \oplus \CC K \oplus \CC d $ be the $ \ell + 2 $-dimensional vector space over $ \CC $ with the (non-degenerate) symmetric bilinear form $ \bl, $ extended from $ \fh  $ by letting
\begin{equation}
\label{e1.02}
(\fh |\CC K + \CC d)=0, \quad (K | K ) = 0 = (d | d), \quad (K | d) = 1.
\end{equation}
\noindent We will identify $ \fh 
 $ with $ \fh^* $ and $ \widehat{\fh} $ with $ \widehat{\fh}^* $ using this form. Given $ \la \in \widehat{ \fh},  $ we denote by $ \bar{\la} $ its orthogonal projection on $ \fh. $
 Let
 \begin{equation}
 \label{e1.03}
 X = \{ h \in \widehat{\fh} \ | \ \Re (K | h) > 0\}.
 \end{equation}
 \noindent For a function $ F $ on $ X $ we will say that it has \textit{degree} $ m $ if 
 \begin{equation}
 \label{e1.04}
 F(h + aK) = e^{ma} F(h) \mbox{ for all } a \in \CC.
 \end{equation}
 
\noindent We shall use the following coordinates on $ \widehat{\fh}: $
\begin{equation}
\label{e1.05}
h = 2 \pi i (- \tau d + z + t K ) =: \tzt, \mbox{ where } \tau, t \in \CC, z \in \fh.
\end{equation}
\noindent Then $ X = \{ (\tau, z , t)  | \Im \tau > 0   \}, $ and $ q:= e^{2 \pi i \tau} = e^{-K}. $

\begin{remark}
\label{r1.01}
One has an obvious bijection between functions on $ X_0 = \{ (\tau, z ) | \Im \tau > 0, z \in \fh   \} $ and functions of degree $ m $ on $ X: $
\[ F(\tau, z) \mapsto F(\tau, z, t) = e^{2 \pi i m t} F(\tau, z), \]
\noindent the converse map being $ F(\tau, z ,t) \mapsto F(\tau, z, 0). $ Adding the factor $ e^{2 \pi i m t} $ simplifies the transformation formulae. Given a function $ F \tzt  $ of degree $ m,  $ we will let $ F(\tau, z) = F(\tau, z, 0)  $ throughout the paper. 
\end{remark}

For $ \al \in \fh  $ define the shift $ p_{\al} $ of $ \widehat{\fh } $ by 
\begin{equation}
\label{e1.06a}
p_{\al} (h) = h + 2 \pi i \al, \quad h \in \widehat{\fh}.
\end{equation}
\noindent Define the following representation of the additive group of the vector space $ \fh $ on the vector space $ \widehat{\fh}: $
\begin{equation}
\label{e1.06}
t_{\al} (h) = h + (K | h) \al - \left(  (\al | h) + \frac{| \al |^2}{2} (K | h) \right)K, \quad \al \in \fh, \ h \in \widehat{\fh}.
\end{equation}
\noindent (We use throughout the shorthand $ | \al |^2 $ for $ (\al | \al). $) This action leaves the bilinear form $ \bl $ on $ \widehat{\fh} $ invariant and fixes $ K, $ hence leaves the domain $ X $ invariant.
Let $ D $ be the Laplace operator on $ \widehat{\fh}, $ associated to the bilinear form $ \bl, $ i.e. $ De^h = | h |^2 e^h,\, h \in \widehat{\fh}. $ In the coordinates (\ref{e1.05}) we have:
\begin{equation}
\label{e1.07}
D = \frac{1}{4 \pi^2} \left( 2 \frac{\partial}{\partial t} \frac{\partial}{\partial \tau} - \sum_{i = 1}^{\ell} \left( \frac{\partial}{\partial z_i} \right)^2    \right),
\end{equation}
\noindent where $ z_1, \ldots, z_{\ell} $ is an orthonormal basis of $ \fh. $

\begin{definition}
\label{d1.02}
Assume that $ L $ is a lattice of rank $ \ell $ (i.e. $ \CC L = \fh $) and that the restriction of the form $ \bl $ to $ L $ is positive definite. For $ m> 0, $ such that (\ref{e1.01}) holds, a \textit{theta function of degree} $ m $ is a holomorphic function $ F $ in the domain $ X $ of degree $ m,  $ such that the following properties hold:
\begin{enumerate}
\item[(i)] $ F(p_{\al}(h)) = F(h) $ for all $ \al \in L; $
\item[(ii)] $ F(t_{\al} (h)) = F(h) $ for all $ \al \in L $;
\item[(iii)] $ DF = 0. $
\end{enumerate}
\end{definition}

\begin{remark}
\label{r1.03}
Due to (\ref{e1.04}) and (\ref{e1.06}),  property (ii) of a function $ F $ of degree $ m $ is equivalent to the property
\begin{enumerate}
\item[$(ii)^\prime $] 
$ F(p_{\tau \al}(h)) = e^{- \pi i m \tau | \al |^2 - m(\al | h)} F(h) $ for all $ \al \in L $.
\end{enumerate}
\end{remark}

\begin{definition}
\label{d1.04}
A \textit{signed theta function} is defined by the same axioms as in Definition \ref{d1.02}, except that (i) and (ii) are replaced by their signed versions:
\begin{enumerate}
\item[$ (i)_- $] $ F(p_{\al}(h)) = (-1)^{m | \al|^2} F(h)$ for all $ \al \in L, $
\item[$ (ii)_- $] $ F(t_{\al} (h)) = (-1)^{m | \al |^2} F(h) $ for all $ \al \in L. $
\end{enumerate}
\noindent (Then $ (ii)' $ is replaced by $ (ii)_{-}^{'} $ with the sign $ (-1)^{m| \al |^2} $ inserted in the RHS.). 
\end{definition}



In order to construct theta functions and signed theta functions, let 
\[ L_{\mathrm{even}} = \{ \al \in L \big{|}\,\, m |\al|^2  \in \ZZ_{\mathrm{even}}  \}, \ L_{\mathrm{odd}} = \{ \al \in L \big{|}\,\, m |\al|^2  \in \ZZ_{\mathrm{odd}}  \}, \]
\noindent and let
\[ \begin{aligned}
P^+_m & = \{ \la \in \widehat{\fh}^* |\,\, (\la | K) = m,\, (\la | L) \subset \ZZ   \}, \\
P^-_m & = \{ \la \in \widehat{\fh}^* |\,\, (\la | K) = m,\, (\la | L_{\mathrm{even}}) \subset \ZZ, \, (\la | L_{\mathrm{odd}}) \subset \half + \ZZ \}.
\end{aligned} \]

\noindent Note that, for $ \la \in \widehat{\fh}, $ such that $ (\la | K) = m, $ among the vectors $ \{ \la - aK |\, a \in \CC \} $ there is a unique isotropic one, for $ a = \frac{|\la|^2}{2m}. $ Hence, in view of axioms (ii) (resp. $ (\mathrm{ii})_- $) and (iii), we construct, for $ \la \in P^+_m $ (resp. $ \la \in P^-_m $) the theta function $ \thl^+ $ (resp. signed theta function $ \thl^- $) by 
\begin{equation}
\label{e1.08}
\thl^{\pm} = e^{-\frac{|\la|^2}{2m}K} \sum_{\al \in L} (\pm 1)^{m|\al|^2} e^{t_{\alpha} (\la)}.
\end{equation}

\noindent The series (\ref{e1.08}) converges to a holomorphic function in the domain $ X $ and obviously satisfies the axioms (i), (ii), (resp. $ \mathrm{(i)}_- $, $ \mathrm{(ii)}_- $), and (iii). Note that
\begin{equation}
\label{e1.09}
\Theta^{\pm}_{\la + m \gamma + aK} = (\pm 1)^{m| \gamma |^2} \Theta^{\pm}_{\la} \mbox{ for } a \in \CC, \gamma \in L.
\end{equation}
\noindent In coordinates (\ref{e1.05}) we have the usual formula, where
$ \la \in P_m^\pm $:
\begin{equation}
\label{e1.10}
\Theta^{\pm}_{\la} (\tau, z, t) = e^{2 \pi i m t}  \sum_{\al \in L} (\pm 1)^{m | \al |^2} 
q^{\frac{m}{2}
| \al + \frac{\bar{\la}}{m} |^2} e^{2 \pi i m
 (\al + \frac{\bar{\la}}{m} | z)}.
\end{equation}
 
 Recall the action of the group $ SL_2 (\RR) $ on the domain $ X $ in coordinates (\ref{e1.05}):
 \begin{equation}
 \label{e1.11}
 \begin{pmatrix}
 a & b \\
 c & d
 \end{pmatrix} \cdot (\tau, z ,t ) = \left( \frac{a \tau + b}{c \tau + d}, \frac{z}{c \tau + d}, t - \frac{c (z | z)}{2(c \tau + d)}
 \right).
 \end{equation}
\noindent This action induces the right action of weight $ w \in \half \ZZ $ of $ SL_2 (\RR) $ on functions in $ X: $
\begin{equation}
\label{e1.12}
F \overset{w}{\bigg{|}}_{\left( \begin{smallmatrix}
a & b \\
c & d \\
\end{smallmatrix} \right)} (\tau, z , t) = (c \tau + d )^{-w} F \left( \begin{pmatrix}
a & b \\
c & d \\
\end{pmatrix} \cdot (\tau, z , t)
\right)
\end{equation}
\noindent (Actually, this is an action of the double cover of $ SL_2(\RR) $ if $ w \in \half + \ZZ. $) The square root of a complex number $ a = re^{i \theta}, $ where $ r \geq 0, - \pi < \theta \leq \pi, $ is, as usual, chosen to be $ a^{\half} = r^{\half} e^{\frac{i \theta}{2}}. $ As usual, we will discuss the action of its subgroup $ SL_2 (\ZZ), $ which is generated by the elements
\[ S = \begin{pmatrix}
0 & -1 \\
1 & 0 \\
\end{pmatrix} \quad \mbox{ and } \quad T = \begin{pmatrix}
1 & 1 \\
0 & 1 \\
\end{pmatrix}. \]

Denote by $ \mathrm{Th}^+_m $ 
(resp. $ \mathrm{Th}_m^- $) 
the space of holomorphic functions on $ X $ of 
degree   $ m$  satisfying properties 
(i), (ii) (resp. $\mathrm{(i)}_-$ , $\mathrm{(ii)}_-$  ), and (iii).

\begin{theorem}
\label{t1.05}
Let $ L $ be a positive definite lattice of rank $ \ell $ and let $ m $ be a positive real number such that (\ref{e1.01}) holds. Then
\begin{enumerate}
\item[(a)] The set of theta functions $ \{ \Theta^+_{\la} \ | \ \la \in P^+_m \mod{(m L + \CC K)} \} $ is a basis of the space $ \mathrm{Th}^+_m. $
\item[(b)] One has the following modular transformation formulas for $ \la \in P^+_m: $
\[ \thl^+ \overset{\ell / 2}{\bigg{|}}_S = e^{-\frac{\pi i \ell}{4}} | L^* / mL |^{- \half} \hspace{-2em} \sum_{\substack{\mu \in \\ P^+_m \! \!\!\! \mod{(mL+ \CC K)}}} \hspace{-2em} e^{- \frac{2 \pi i}{m} (\bar{\la}| \bar{\mu})} \Theta_{\mu}^+. \]
\item[(c)] Provided that $ m(\al | \al) \in 2 \ZZ $ for all $ \al \in L, $ one has
\[ \thl^+ \overset{\ell / 2}{\bigg{|}}_T = e^{\frac{\pi i | \bar{\la} |^2}{m}} \thl^+, \]
\noindent hence the space $ \mathrm{Th}^+_m $ is $ SL_2(\ZZ) $-invariant. 
\end{enumerate}
\end{theorem}

\begin{proof}
The key formula for the proof is 
\begin{equation}
\label{e1.13}
(DF)
\overset {w}{\bigg{|}}_A
= (c \tau + d)^2 D(F\overset{w} {\bigg{|}}_A),\,\, \ A = \begin{pmatrix}
a & b \\
c & d \\
\end{pmatrix} \in SL_2 (\RR),\, 
\hbox{if}\,\, w=\frac{\ell}{2}.
\end{equation}
\noindent See \cite{KP} or \cite{K2}, Proposition 13.3, Lemma 13.2, and Theorem 13.5, for details. 
\end{proof}

\begin{theorem}
\label{t1.06}
Let $ L $ and $ m$ be as in Theorem \ref{t1.05}, and assume that $ L_{\mbox{odd}} \neq \emptyset. $ Then
\begin{enumerate}
\item[(a)] The set $ \{ \thl^- \ | \ \la \in P^-_m \mod{(mL + \CC K)} \} $ is a basis of the space $ \mathrm{Th}_m^-. $
\item[(b)] One has for $ \la \in P^-_m $ the same modular transformation formulas for $ \thl^-, \la \in P^-_m, $ where $ P^+_m $ is replaced by $ P^-_m, $
 \end{enumerate}
\noindent hence the space $ \mathrm{Th}^-_m $ is $ SL_2 (\ZZ) $-invariant.
\end{theorem}

\begin{proof}
Uses reduction to Theorem 1.5, see \cite{KW4}, Proposition A3.
\end{proof}

\begin{remark}
\label{r1.07}
In view of (\ref{e1.10}) one sometimes uses a slightly different notation $ \Theta^{\pm}_{\bar{\la}, m} $ for (\ref{e1.08}).
\end{remark}

\begin{example}
\label{x1.08}
Let $ L = \ZZ $ with the bilinear form $ (a| b) = 2ab, $ so that $ L^* = \half \ZZ.  $ Then for a positive integer $ m $ we have the following basis of $ \mathrm{Th}^+_m \,\,\, (\tau, z, t \in \CC, \Im \tau > 0):$
\[ \Theta_{j,m}^+ \tzt = e^{2 \pi i m t} \sum_{n \in \ZZ}  q^{m(n + \frac{j}{2m})^2} e^{2 \pi i m( n + \frac{j}{2m})z}, \quad j = 0, 1, \ldots, 2m-1. \]
\noindent The modular transformation formulae are:
\[ \begin{aligned}
\Theta_{j,m}^+ \left( - \frac{1}{\tau}, \frac{z}{\tau}, t - \frac{z^2}{2 \tau} \right) & = \left(  \frac{-i \tau}{2m}\right)^{\frac{1}{2}}\,\,  \sum_{j' = 0}^{2m-1} e^{- \frac{\pi i jj'}{m}} \Theta^+_{j', m} \tzt, \\
\Theta^+_{j,m} \left( \tau + 1, z ,t \right) & = e^{\frac{\pi i j^2}{2m}} \Theta^+_{j,m} \tzt.
\end{aligned} \]
\end{example}

\begin{example}
\label{x1.09}
Let $ L $ and $(.|.)$ be the same as in Example \ref{x1.08}. Then for  $ m \in \frac{1}{2} +\ZZ_{\geq 0} $ we have the following basis of $ \mathrm{Th}_m^-: $
\[ \Theta^-_{j+ \half,m} \tzt = e^{2 \pi i m t} \sum_{n \in \ZZ} (-1)^n q^{m \left( n + \frac{2j + 1}{4m}\right)^2} e^{4 \pi i m \left( n + \frac{2 j + 1}{4m}\right)z}, \quad j = 0, 1, \ldots, 2m-1.  \]
\noindent The modular transformation formulae are ($ i = 0, 1, \ldots, 2m-1 $):
\[ \begin{aligned}
\Theta^-_{  j + \half,m} \left(  -\frac{1}{\tau}, \frac{z}{\tau}, t - \frac{z^2}{2 \tau}  \right) & = \left( - \frac{i \tau}{m}\right)^{\half}\,\,\, \sum_{j' = 0}^{2m-1} e^{-\frac{ \pi i}{m} (j + \half) (j' + \half)} \Theta^-_{j'+ \half,m} \tzt, \\
\Theta^-_{j + \half,m} (\tau + 1, z, t) & = e^{\frac{\pi i}{2m} (j + \half)^2} \Theta^-_{j + \half,m} \tzt.
\end{aligned}
 \]
 \noindent Note that $ \mathrm{Th}^-_{\frac{1}{2}} $ is 1-dimensional, and is spanned by $ \Theta^-_{\half , \half} \tzt = -ie^{\pi i t} 
\vartheta_{11} (\tau, 2z) $.

One has the following elliptic transformation formulas:
\begin{equation}
 \begin{split}
\label{e1.15}
\Theta^\pm_{j,m} (\tau, z+a, t) & = e^{\pi i j a} \Theta^\pm_{j,m} \tzt \mbox{ if } am \in \ZZ, \\
\ \Theta^\pm_{j,m} \left(\tau, z + \frac{\tau}{m}, t \right) & = q^{-\frac{1}{4m}} e^{-\pi i z} \Theta^\pm_{j+1, m} \tzt. \\
\end{split} 
\end{equation}
\end{example}

\begin{remark}
\label{r1.10} For each $ m \in \frac{1}{2}\zp, $ all the function $ \Theta^{\pm}_{j,m}, \ j \in \half \ZZ, $  span a finite dimensional $ SL_2(\ZZ) $-invariant space, whose modular and elliptic transformation properties follow from \cite{KW4}, Proposition A3. Especially important are the four Jacobi forms
\[ \begin{aligned}
\vartheta_{00} (\tau, z) & = \Theta^+_{0, \half} (\tau, 2z,0), \quad \vartheta_{01} (\tau, z) = \Theta^-_{0, \half} (\tau, 2z,0), \\
\vartheta_{10} (\tau, z) & = \Theta^+_{\half, \half} (\tau, 2z,0), \quad \vartheta_{11} (\tau, z) = i \Theta^-_{\half, \half} (\tau, 2z,0),  \
\end{aligned} \]
\noindent which span an $ SL_2 (\ZZ) $-invariant subspace of functions, holomorphic in $ X_0 $ (see \cite{M}, p 36, or \cite{KW3}, Proposition A7 for the modular transformation formulas).
\end{remark}

\begin {remark}
\label{r1.11}  
For an arbitrary $w\in \half \ZZ$ one should add the following term in the RHS of (\ref{e1.13}): $2c(\frac{\ell}{2}-w)(c\tau +d)^{-1}(\frac{dF}{d\tau})\overset {w}{\bigg{|}}_A$.
\end{remark}

\section{Mock theta functions.}
In this Section we keep notation and assumptions of Section 1, except for the assumption that $ \CC L = \fh $ in Definition \ref{d1.02} (resp. \ref{d1.04}) of a theta (resp. signed theta) function. 

Let $ B \subset \fh $ be a linearly independent set of vectors, such that the following two properties hold:
\begin{equation}
\label{e2.01}
(B|B) = 0 \mbox{ and } \fh = \CC L \oplus \CC B.
\end{equation}
\noindent Let
\begin{equation}
\label{e2.02}
P^{\pm}_{m;B} = \{ \la \in P^{\pm}_m  \ \big{|} \ (\la | B) = 0 \}
\end{equation}

\noindent For $ \la \in P^{\pm}_{m;B} $ we construct the degree $ m $ \textit{mock theta function } $ \Theta^+_{\la;B} $ and signed \textit{mock theta function} $ \Theta^-_{\la;B} $ as follows [KW3--5], cf. (\ref{e1.08}):
\begin{equation}
\label{e2.03}
\Theta^{\pm}_{\la; B} = e^{-\frac{|\la|^2}{2m}} \sum_{\la \in L} (\pm 1)^{m |\al|^2} t_{\al} \frac{e^{\la}}{\prod_{\beta \in B} (1-e^{- \beta})}.
\end{equation}

\noindent This series converges to a meromorphic function in the domain $ X,$ and in the coordinates (\ref{e1.05}) it looks as follows:
\begin{equation}
\label{e2.04}
\Theta^{\pm}_{\la;B} \tzt = e^{2 \pi i m t} \sum_{\al \in L} (\pm 1)^{m|\al|^2} \frac{q^{\frac{m}{2}|\al + \frac{\bar{\la}}{m}|^2} e^{2 \pi i m (\al + 
\frac{\bar{\la}}{m}| z)}}{\prod_{\beta \in B} \left( 1 - q^{- (\al + 
\frac{\bar{\la}}{m}| \beta)} e^{-2 \pi i (\beta | z)} \right)}.
\end{equation}

\noindent Note that for mock theta functions we have an analogue of (\ref{e1.09}):
\begin{equation}
\label{e2.05}
\Theta^{\pm}_{\la + m \gamma + a K; B} = (\pm 1)^{m|\gamma|^2} \Theta^{\pm}_{\la; B} \mbox{ for } a \in \CC, \gamma \in L.
\end{equation}
\noindent Moreover, these functions satisfy all axioms
$ (i), (ii) $ (resp. $ (i)_-, (ii)_- $) and $ (iii) $ for $ + $ (resp. $ - $)
of Definition \ref{d1.02} (resp. \ref{d1.04}). 
However, not being holomorphic, they are not members of $ \mathrm{Th}^+_m $ (resp. $ \mathrm{Th}^-_m $)

\begin{example}
\label{x2.01}
Let $ \fh $ be a 2-dimensional vector space with basis $ \al_1, \al_2 $ and the symmetric bilinear form $ \bl $, given by 
\[ |\al_i|^2 = 0, \quad i = 1,2, \quad (\al_1 | \al_2) = 1. \]
\noindent Let $ \theta = \al_1 + \al_2, $ so that $ |\theta|^2 = 2, $ let $ L = \ZZ \theta$ and $ B = \{ \beta = \al_1 \}. $ Let $ m $ be a positive integer. Then 
\[ P^+_{m;B} = \{ \la^+_{m,s} = md + s \al_1 \ \big{|} \ s \in \ZZ \} + \CC K, \]
\noindent and in coordinates
\begin{equation}
\label{e2.06}
h = 2 \pi i (- \tau d - z_1 \al_2 - z_2 \al_1 + tK) = :(\tau, z_1, z_2, t)
\end{equation}
\noindent we have:
\[ \Theta^+_{\la^+_{m,s};B} = \Phi^{+[m,s]} (\tau, z_1, z_2, t), \]
\noindent where 
\begin{equation}
\label{e2.07}
\Phi^{+[m,s]} (\tau, z_1, z_2, t) = e^{2 \pi i m t} \sum_{n \in \ZZ}  \frac{ q^{mn^2+ns}  e^{2 \pi i (mn(z_1 + z_2) +sz_1)} }{1 - e^{2 \pi i z_1} q^n}. 
\end{equation}

\noindent Now, let $ m \in \half + \ZZ_{\geq 0}. $ Then
\[ P^-_{m;B} = \{ \la^-_{m,s} = md + s 
\al_1 \  \big{|} \ s \in \thalf + \ZZ \} + \CC K, \]
\noindent and in coordinates (\ref{e2.06}) we have
\[ \Theta^-_{\la^-_{m,s}; B} = \Phi^{-[m,s]} (\tau, z_1, z_2, t), \]
\noindent where 
\begin{equation}
\label{e2.08}
\Phi^{-[m,s]} (\tau, z_1, z_2, t) = e^{2 \pi i m t} \sum_{n \in \ZZ} (-1)^n \ \frac{ q^{mn^2+ns}  e^{2 \pi i (mn(z_1 + z_2) +sz_1)} }{1 - e^{2 \pi i z_1} q^n}.
\end{equation}
\end{example}

Unlike the (signed) theta functions $ \thl^{\pm}, $ the (signed) mock theta functions $ \Theta^{\pm}_{\la;B} $ with $ B \neq \emptyset, $ have neither good elliptic transformation properties analogous to (i) and (ii)  
(resp. $(\mathrm{{i}})_- $ and $ (\mathrm{ii})_- $) of Definition \ref{d1.02} (resp. \ref{d1.04}), nor good modular transformation properties, given by Theorem \ref{t1.05} (resp. \ref{t1.06}). It was S. Zwegers, 
who showed in Chapter 1 (resp. 3) of his beautiful paper \cite{Z} 
that, upon adding (in our notation) to 
$ \Phi^{-[\frac{1}{2},\frac{1}{2}]} $  (resp. $ \Phi^{+[m,0]} $) 
a real analytic, but not holomorphic, ``modifier'', the modified function 
does satisfy beautiful elliptic and modular transformation properties. His results were extended to all functions $ \Phi^{\pm[m,s]}, $ described in Example \ref{x2.01}, in \cite{KW3} and \cite{KW4}.

Moreover, we showed in \cite{KW5} that, under some assumptions on $ L $ and $ B, $ the modified mock theta function $ \tilde{\Theta}^{\pm}_{\la;B}, $ by making use of the function $ \tilde{\Phi}^{\pm[m,s]}, $ satisfy nice elliptic and modular transformation properties as well. We used this in \cite{KW5} to show that the character of ``tame integrable'' modules over affine Lie superalgebras have nice modular transformation properties. 

Let us recall the construction of the modified mock theta function $ \tilde{\Phi}^{\pm[m,s]} $ and their transformation properties from \cite{Z}, \cite{KW3}, \cite{KW4}.

For $ x \in \RR, \ t, z \in \CC, \ \Im \tau > 0, \ m \in \half \zp, \ $ and $ n \in \half \ZZ $ let 
\[ E(x) = 2 \int_{0}^{x} e^{-\pi u^2} \ du, \mbox{ and } \psi_{m,n} (\tau, z) = \left(n-2m \frac{\Im z}{\Im \tau}\right) \sqrt{\frac{\Im \tau}{m}}. \]
\noindent For $ m \in \half \zp, \ j \in \half \ZZ $ let 
\begin{equation}
\label{e2.09}
R^{\pm}_{j,m} (\tau, z) = \hspace{-1.5em} \sum_{\substack{ n \in \half \ZZ \\ n \equiv j \! \mod 2 m}} \hspace{-1.5em} (\pm 1)^{\frac{n-j}{2m}} \left(  \mathrm{sign} (n - \thalf - j + 2m) - E(\psi_{m,n} (\tau,z) )        \right) e^{-\frac{\pi i n^2}{2m} \tau + 2 \pi i n z},
\end{equation} 
\noindent and define the \textit{modifier}
\begin{equation}
\label{e2.10}
\Phi^{\pm[m,s]}_{\mathrm{add}} (\tau, z_1, z_2, t) = e^{2 \pi i  mt} \hspace{-1em} \sum_{\substack{j \in s + \ZZ \\ s \leq j < s + 2m}} \hspace{-1em} R^{\pm}_{j,m} \left( \tau, \frac{z_1-z_2}{2}  \right) \Theta^{\pm}_{j,m} (\tau, z_1 + z_2).
\end{equation}

\noindent For $ m \in \half \zp $ and $ s \in \half \ZZ $ define the \textit{modified mock theta function}
\begin{equation}
\label{e2.11}
\tilde{\Phi}^{\pm[m,s]} = \Phi^{\pm [m,s]} - \half \Phi^{\pm [m,s]}_{\mathrm{add}}.
\end{equation}

\noindent It is proved in \cite{KW4}, Corollary 1.6(a), that $ \tilde{\Phi}^{\pm[m,s]} $ remains unchanged if we add  to s  an integer. 

As in [KW3--5], consider the following change of coordinates:
\begin{equation}
\label{e2.12}
z_1 = v - u, \quad z_2 = -v-u, \quad \mbox{i.e. } u = -\frac{z_1 + z_2}{2}, \quad v = \frac{z_1 - z_2}{2}, 
\end{equation}
\noindent and let
\begin{equation}
\label{e2.13}
\varphi^{\pm [m,s]} (\tau, u, v, t) = \Phi^{\pm[m,s]} (\tau, z_1, z_2, t),\,\,
\tilde{\varphi}^{\pm [m,s]} (\tau, u, v, t) =\tilde{\Phi}^{\pm[m,s]} (\tau, z_1, z_2, t).
\end{equation}

The action (\ref{e1.11}), (\ref{e1.12}) of $ SL(2, \RR) $ for $ w=1 $ on functions on $ X $ in the coordinates $ (\tau, u, v, t) $ looks as follows:
\begin{equation}
\label{e2.14}
\varphi \ \big{|}_{\left( \begin{smallmatrix}
a & b \\
c & d \\
\end{smallmatrix} \right)} (\tau, u, v, t) = (c \tau + d)^{-1} \ \varphi \left( \frac{a \tau + b}{c \tau +d}, \frac{u}{c \tau +d}, \frac{v}{c \tau +d}, t - \frac{c(u^2 - v^2)}{c \tau +d} \right).
\end{equation}
\noindent Here and further on we skip $ w $ from notation (\ref{e1.12}) if
$w=1$.

We will use two coordinate systems of the vector space $ \CC \tau + \CC v, $ viewed as a 4-dimensional vector space over $ \RR: $
\begin{equation}
\label{e2.15}
\tau, \,\bar{\tau},\, v,\, \bar{v}
\end{equation}
\noindent and
\begin{equation}
\label{e2.16}
\tau=x+i y,\,\bar{\tau}=x-iy,\, a,\, b, \,\mbox{ where }  \ v = a \tau - b.
\end{equation}
\noindent We have:
\begin{equation}
\label{e2.17}
2x = \tau + \bar{\tau}, \quad 2iy = \tau - \bar{\tau}, \quad 2iya = v - \bar{v}, \quad 2iyb = \bar{\tau} v - \tau \bar{v}.
\end{equation}
\noindent Hence
\begin{equation}
\label{e2.18}
\nabla :=
\frac{\partial}{\partial a} + \tau \frac{\partial}{\partial b} = -2iy \frac{\partial}{\partial \bar{v}}.
\end{equation}
When writing $ \frac{\partial}{\partial \tau}, \frac{\partial}{\partial \bar{\tau}}, \frac{\partial}{\partial v}, \frac{\partial}{\partial \bar{v}} $ (resp. $ D_{\tau}, D_{\bar{\tau}},  \frac{\partial}{\partial a}, \frac{\partial}{\partial b}$), 
we mean partial derivatives with respect to coordinates (\ref{e2.15}) (resp. (\ref{e2.16})).
As in \cite{Z} and \cite{KW4}, coordinates (\ref{e2.16}) play an important role in what follows. 

Introduce the following differential operators:
\begin{equation}
\label{e2.19}
D = 4 \frac{\partial}{\partial t} \frac{\partial}{\partial \tau}+ \left(\frac{\partial}{\partial v} \right)^2 - \left(\frac{\partial}{\partial u}\right)^2, \quad \bar{D} = 
2 \frac{\partial}{\partial t} 
D_{\bar{\tau}} +\left(\frac{\partial}{\partial \bar{v}}\right)^2,
 \quad \Delta = 2a \frac{\partial}{\partial t} \frac{\partial}{\partial \bar{v}} - \frac{\partial}{\partial v} \frac{\partial}{\partial \bar{v}}.
\end{equation}
\noindent Note that, up to a constant factor, $ D $ is the Laplace operator, associated to the bilinear form of Example \ref{x2.01}. We have formulas, similar to (\ref{e1.13}), for $ A \in SL_2 (\RR): $
\begin{equation}
\label{e2.20}
(\bar{D}F) {\big{|}}_A = (c \bar{\tau} +d)^2  \bar{D} 
(F {\big{|}}_A), \quad (\Delta F) {\big{|}}_A = 
|c\tau +d|^2 \Delta (F{\big{|}}_A) \ .
\end{equation}

Let $ m \in \half \zp $ and $ s, s' \in \half \ZZ.  $ Let $ \cF^{[m; s, s']} $ be the space of functions $ F $ on the domain $X$, i.e. functions in $ \tau, u, v, t \in \CC^4, $ such that $ \Im \tau > 0 $,  satisfying the following five conditions: 
\begin{enumerate}
\item[(F1)] $ F(\tau, u, v, t) = e^{2 \pi i m t} \, F(\tau, u,v, 0)$, i.e.
$F$ is a function of degree $m$;
\item[(F2)] $ \vartheta_{11} (\tau, v-u) \vartheta_{11} (\tau, v+u) \, F (\tau, u, v, t)  $ is a holomorphic function in $ u $ and real analytic in $ \tau  $ and $ v; $ $\frac{\partial F}{\partial\bar{v}}(\tau, u, v, t)$ is a real analytic function in all variables;
\item[(F3)] for all $ j,k \in \ZZ  $ one has:
\begin{enumerate}
\item[(i)] $ F(\tau, u+j, v+k, t) = e^{2 \pi i s(j+k)} \, F(\tau, u, v, t), $
\item[(ii)] $ F(\tau, u+ j \tau, v + k \tau, t) = e^{2 \pi i (s' (j+k) + 2m (kv-ju))} q^{m (k^2 - j^2) } \, F(\tau, u, v, t); $
\end{enumerate}
\item[(F4)] for all $ k \in \ZZ $ one has:
\begin{enumerate}
\item[(i)] $ F \left( \tau, u + \frac{k}{2m}, v + \frac{k}{2m}, t \right) = F(\tau, u, v, t), $
\item[(ii)] $ F \left( \tau, u + \frac{k \tau}{2m}, v + \frac{k \tau}{2m}, t  \right) = e^{2 \pi i k (v-u)} \, F (\tau, u, v, t).$
\end{enumerate}
\item[(F5)] $ DF = 0, \ \bar{D}F = 0, \ \Delta F = 0. $
\end{enumerate}

The main results of the paper are the following three theorems. 

\begin{theorem}
\label{t2.2}
Provided that $ m > \half $, and either $s$ or $s'$ is not an integer  if  $m = 1$, one has 
\[ \cF^{[m; s, s']}=\CC  \tilde{\varphi}^{+[m,s]}\, (\hbox{resp.} = \CC\tilde{\varphi}^{-[m,s]} )\,\,\hbox{ if}\,\,  s' \in \ZZ \,\, (\hbox{resp.}\,\,  s' \in \half + \ZZ ).
\]
\end{theorem}

\begin{theorem}
\label{t2.3}
\begin{enumerate}
\item[(a)] If $ m = 1 $ and $ s,s' \in \ZZ, $ then 
\[ \cF^{[1;s,s']}  = \CC \tilde{\varphi}^{+[1,0]} \oplus \CC \widehat{R}^A,\] 
where 
 \begin{equation}
 \label{e2.21}
 \widehat{R}^A (\tau, u, v, t) = e^{2 \pi i t} \frac{\eta(\tau)^3 \vartheta_{11} (\tau, 2 u)}{\vartheta_{11} (\tau, v-u) \vartheta_{11}(\tau, v + u)}.  
 \end{equation}
\item[(b)] If $ m = \half, $ then 
\[ \cF^{[\half; s, s']} = \CC \tilde{\varphi}^{+[\half, s]} \oplus \CC \widehat{R}^B_{2s \bmod 2 \ZZ,  0}\,\, \hbox{if}\,\,  s' \in \ZZ, \]
and 
\[ \cF^{[\half; s, s']} = \CC \tilde{\varphi}^{-[\half, s]} \oplus \CC \widehat{R}^B_{2s \bmod 2\ZZ, 1} \,\,\hbox{if}\,\,  s' \in \half + \ZZ,  \]
where 
 \begin{equation}
 \label{e2.22}
 \widehat{R}^B_{ab} (\tau, u, v, t) = e^{\pi i t} \widehat{R}^A (\tau, u , v, 0) \, \frac{\vartheta_{ab} (\tau, v)}{\vartheta_{ab} (\tau, u)},\,\, a,b = 0\, 
\hbox{or}\, 1 . 
 \end{equation}
\end{enumerate}
\end{theorem}

\begin{theorem}
\label{t2.4}
If $F\in \cF^{[m;s,s']} $
, then 
\[F\big{|}_S\in 
\cF^{[m;s',s]} 
,\,\,\hbox{and}\,\,  F\big{|}_T\in \cF^{[m;s,s+s'+m]} \]     
with respect to the action (\ref{e2.14}). Consequently $\cF^{[m;s,s']} $
is $SL_2(\ZZ)$-invariant, provided that $m \equiv s \equiv s'\mod\ZZ$.  
\end{theorem}

\begin{remark}
\label{r2.5}
The function $ \widehat{R}^A $ is the superdenominator for the affine Lie superalgebra $ \widehat{s \ell}_{2|1}, $ and the functions $ \widehat{R}^B_{ab} $ 
are the untwisted and twisted denominator and superdenominator for $ \widehat{\osp}_{3|2}, $ see \cite{KW4}. Since, by the product decompositions of the function $\vartheta_{ab}$ (see e.g. \cite{M} or \cite{KW4}), one has 
\[ \prod_{a,b = 0,1} \vartheta_{ab} (\tau, u) = \eta (\tau)^3 \vartheta_{11} (\tau, 2 u), \]
\noindent the functions $ \widehat{R}^B_{ab} $ satisfy condition (F2).
\end{remark}
  \begin{remark}
 \label{r2.6}
 Property (F3) is equivalent to the following two properties: 
\[  F(p_{\al}(h)) = (-1)^{s|\al|^2}F(h)  \mbox{ and }  F(t_{\al} (h)) = (-1)^{s' |\al|^2} F(h)\,\, \mbox{ for all }\,\,\al= n_1 \al_1 + n_2 \al_2,  \]
 where $n_1, n_2\in \ZZ$ and
$ \al_1 $,  $ \al_2 $ are as in Example \ref{x2.01}, cf. Remark \ref{r1.03}. 
Also, property (F4) is equivalent to the following:
\[  F(p_{\al}(h)) = F(h)  \mbox{ and }  F(t_{\al} (h)) = F(h)\,\, \mbox{ for all }\,\,\al=\frac{n}{m} \al_1, \ n\in \ZZ .  \]
Thus, the modified mock theta functions have more symmetries than the mock theta functions.
 \end{remark}
  \begin{remark}
 \label{r2.7}
 It follows from Theorems \ref{t2.2} and \ref{t2.3} that if 
$ f \in \cF^{[m;s,s']} $  is a real analytic function in $ \tau, u, v, $ then $ f = 0. $ Indeed the functions $ \tilde{\varphi}^{\pm[m,s] } \tuvt$ (resp. $ \widehat{R}^A \tuvt $ and  $ \widehat{R}_{ab}^B \tuvt $) have singularities at the points $ v-u \in \ZZ + \tau \ZZ $ (resp. $ v \pm u \in \ZZ + \tau \ZZ $). It follows that if $ \varphi \tuvt $ is a function of degree $ m $ and $ f_i \in \cF^{[m;s,s']}, \ i = 1,2, $ are such that $ f_i - \varphi $ are real analytic functions, then $ f_1 = f_2. $
 \end{remark}
  \begin{remark}
 \label{r2.8}
 In Definition 3.2 of \cite{Z}, Zwegers introduces function $ f_v (u; \tau ) $ and in Definition 3.4 their modifications $ \tilde{f}_v (u; \tau). $ One has:
 \[ e^{2 \pi i m t} f_v (u;\tau) = -\varphi^{+[m,1]} \tuvt. \]
 \noindent Furthermore, it follows from Remark \ref{r2.7} that
 \[ e^{2 \pi i m t} \tilde{f}_v (u;\tau) = - \tilde{\varphi}^{+[m,1]} \tuvt \ ( = - \tilde{\varphi}^{+[m,0]} \tuvt). \]
 \end{remark}
  \begin{remark}
 \label{r2.9}
 If we replace condition (F2) by a stronger condition 
 \begin{enumerate}
 \item[$ (\mathrm{F2})' $] $ \vartheta_{11} (\tau, v-u) F \tuvt $ is holomorphic in $ u $ and real analytic in $ \tau  $ and $ v, $ and $\frac{\partial F}{\partial \bar{v}}$ is a real analytic function in all variables,
 \end{enumerate}
 \noindent then the corresponding space of functions $ \cF'^{[m;s.s']} $ is spanned by $ \tilde{\varphi}^{\pm[m,s]} \tuvt, $ where $ + $ (resp. $ - $) corresponds to $ s' \in \ZZ $ (resp. $ \in \half + \ZZ $).
 \end{remark}


\begin{remark}
\label{r2.10}
The operator $ D_{\bar{\tau}} = \frac{\partial}{\partial \bar{\tau}} +a\frac{\partial}{\partial \bar{v}} $ satisfies the following commutation relations (recall that $ a = \frac{v - \bar{v}}{\tau - \bar{\tau}} $):
\[\left [ D_{\bar{\tau}}, \frac{\partial}{\partial \tau}\right] = \frac{a}{\tau - \bar{\tau}}\frac{\partial}{\partial \bar{v}}, \quad 
\left[ D_{\bar{\tau}}, \frac{\partial}{\partial \bar{\tau}}\right ] = - \frac{a}{\tau - \bar{\tau}}\frac{\partial}{\partial \bar{v}}, \ 
\left[D_{\bar{\tau}}, \left( \frac{\partial}{\partial v}\right)^2 \right] = \frac{-2}{\tau - \bar{\tau}} \frac{\partial}{\partial v} \frac{\partial}{\partial \bar{v}}, \]
\[ \left[D_{\bar{\tau}}, 
\left( \frac{\partial}{\partial \bar{v}}\right)^2 \right] = \frac{2}{\tau - \bar{\tau}} \left( \frac{\partial}{\partial \bar{v}}\right)^2 \ , \,
 \left[ D_{\bar{\tau}}, \frac{\partial}{\partial v} \frac{\partial}{\partial \bar{v}}\right] = \frac{1}{\tau - \bar{\tau}} \frac{\partial}{\partial \bar{v}} \left( \frac{\partial}{\partial v } - \frac{\partial}{\partial \bar{v}} \right) \ . \]
From these formulas we obtain the following commutation relations on the space of functions of degree $ m: $
\[ [\bar{D}, D] = \frac{16 \pi i m}{\tau - \bar{\tau}} \Delta, \quad [\Delta, D] = \frac{8 \pi i m }{\tau - \bar{\tau}} \Delta, \quad [\Delta, \bar{D}] = - \frac{8 \pi i m }{\tau - \bar{\tau}} \Delta \ . \]
It follows that condition (F5) can be replaced by 
\begin{enumerate}
\item[(F5)$'$] $ DF = 0, \ \bar{D}F = 0. $ 
\end{enumerate}
\end{remark}
 
While the proof of Theorems \ref{t2.2} and \ref{t2.3} is rather involved, and will be given in the next sections, the proof of Theorem \ref{t2.4} is straightforward from the definition of the spaces $\cF^{[m;s,s']}$, and it is given here.  
\begin{proof}[Proof of Theorem 2.4]
By equation (\ref{e2.14}) we have for $ F $ of degree $ m: $
\begin{equation}
\label{e2.23}
F \big{|}_S (\tau, u,v,t) = \frac{1}{\tau} e^{\frac{2 \pi i m}{\tau} (v^2 - u^2)} F \left( -\frac{1}{\tau}, \frac{u}{\tau}, \frac{v}{\tau}, t \right)
\end{equation}
\noindent Using this, condition (F3)(i) gives
\begin{equation}
\label{e2.24}
F \big{|}_S (\tau, u + j \tau, v + k \tau, t) = e^{2 \pi i s (j+k)} e^{4 \pi i m (kw - ju)} q^{m(k^2 - j^2)} F \big{|}_S (\tau, u, v, t)
\end{equation}
Next, replacing $ j,k $ by $ -j, -k $ and $ (\tau, u ,v ) $ by $ (-\frac{1}{\tau}, \frac{u}{\tau}, \frac{v}{\tau}) $ in (F3)(ii), we obtain:
\begin{equation}
\label{e2.25}
F \big{|}_S (\tau, u +j, v+ k, t) = e^{2 \pi i s' (j+k)} F \big{|}_S (\tau, u, v, t).
\end{equation}
\noindent Formulas (\ref{e2.24}) and (\ref{e2.25}) mean that if $ F  $ satisfies conditions (F1) and (F3), then $ F \big{|}_S $ satisfies (F3) for $ \cF^{[m;s',s]}. $

Similarly one shows that if $ F $ satisfies conditions (F1) and (F4), then $ F \big{|}_S $ satisfies condition (F4). Also, by (\ref{e1.13}) and (\ref{e2.20}), $ F \big{|}_S $ satisfies condition (F5) if $ F $ does, proving the first part of the theorem. 

In order to prove the second part, it suffices to show that if $ F(\tau, u, v, t) $ of degree $ m $ satisfies condition (F3)(ii), then $ F (\tau + 1, u, v, t) $ satisfies that condition with $ s' $ replaced by $s+ s' + m. $ This is straightforward. 
\end{proof}

 \section{Proof of Theorems \ref{t2.2} and \ref{t2.3}.}
 In order to prove Theorem \ref{t2.2} and \ref{t2.3}, consider the space $ \cG^{[m;s,s']}   $ of functions $ G(\tau, u,v,t), $ satisfying the following conditions:
 \begin{enumerate}
 \item[(G1)] $ G(\tau, u,v, t) = e^{2 \pi i m t} G(\tau, u, v, 0); $
 \item[(G2)] $ G(\tau, u, v, 0) $ is a holomorphic function in $ u $ and $ \bar{v}, $ and real analytic in $ \tau,\, \Im \tau > 0; $
 \item[(G3)] for all $ j, k \in \ZZ  $ one has
 \begin{enumerate}
 \item[(i)] $ G(\tau, u+j, v+k, t ) = e^{2 \pi i s (j+k)} G(\tau, u, v, t), $
 \item[(ii)] $ G(\tau, u + j \tau, v + k \tau, t) = e^{2 \pi i (s' (j+k) + 2m (k \bar{v}-ju))} \ e^{2 \pi i m (k^2 \bar{\tau}- j^2 \tau )} G(\tau, u, v, t);  $
 \end{enumerate}
 \item[(G4)] for all $ k \in \ZZ $ one has:
 \begin{enumerate}
 \item[(i)] $ G \left( \tau, u + \frac{k}{2m}, v + \frac{k}{2m}, t  \right) = G(\tau, u, v, t), $
 \item[(ii)] $ G \left( \tau, u + \frac{k \tau}{2m}, v + \frac{k \tau}{2m}, t  \right) = e^{2 \pi i k (\bar{v} - u)} \ e^{\frac{\pi i k^2}{2m} (\bar{\tau} - \tau)} G(\tau, u, v, t). $
 \end{enumerate}
 \item[(G5)] $ DG = 0, \quad \left( 4 \frac{\partial }{\partial t} \frac{\partial }{\partial \bar{\tau}} + \left(\frac{\partial}{\partial \bar{v}}\right)^2  \right) G = 0. $
 \end{enumerate}
 
 Next, we construct a linear map $ F \mapsto \theta_F $ from the space $ \cF^{[m;s,s']} $ to the space $ \cG^{[m;s,s']}, $ show that $ \tilde{\varphi}^{\pm[m,s]} $ lies in $ \cF^{[m;s,s']}, $ where $ + $ (resp. $ - $) corresponds to $ s' \in \ZZ $ (resp. $ \in \half + \ZZ $), and that the image of this function spans $ \cG^{[m;s,s']}. $ The last step of the proof is the study of the kernel of this map.
 
 Given a differentiable in $ \bar{v} $ function $ F = F(\tau, u, v, t) $ of degree $ m $ in the domain $ X $ (i.e. $ \tau, u, v, t \in \CC, \Im \tau > 0 $), let 
 \begin{equation}
 \label{e3.1}
 \theta_F (\tau, u,v,t) = y^{-\half} e^{4 \pi  m a^2 y} \nabla F = -2iy^{\half} e^{4 \pi  m a^2 y} \frac{\partial F }{\partial \bar{v}},
 \end{equation}
 \noindent where $ \nabla $ is the operator, defined in (\ref{e2.18}). We shall prove that
 \begin{equation}
 \label{e3.2}
  \theta_{\tilde{\varphi}^{\pm [m,s]}} =-2 \sqrt{m} \ \theta^{\pm[m,s]}, 
 \end{equation}
 where
 \begin{equation}
 \label{e3.3}
 \theta^{\pm[m,s]} (\tau, u, v, t) = e^{2 \pi i m t} \hspace{-2em}\sum_{j \in s + \ZZ \bmod 2m \ZZ} \hspace{-2em} \Theta^{\pm}_{j,m} (- \bar{\tau}, 2 \bar{v}) \Theta^{\pm}_{-j,m} (\tau, 2u).
 \end{equation}
 
 Theorems \ref{t2.2} and \ref{t2.3} follow easily from the following lemmas, where $ m \in \half \zp, \ s, s' \in \half \ZZ.  $
 \begin{lemma}
 \label{L3.1}
 For $ F \in \cF^{[m;s,s']} $ we have: $ \theta_F \in \cG^{[m;s,s']}. $
 \end{lemma}
 
 \begin{lemma}
 \label{L3.2}
 \begin{enumerate}
 \item[(a)] $ \cG^{[m;s.s']} = \CC \theta^{+[m,s]} $ (resp. $=\CC  \theta^{-[m,s]} $) if $ s' \in \ZZ $ (resp. $ \in \half + \ZZ $),
 \item[(b)]  $ \tilde{\varphi}^{\pm [m,s]}  \in \cF^{[m;s,s']},$ where $ + $ (resp. $ - $) corresponds to $ s' \in \ZZ $ (resp. $ s' \in \half + \ZZ $).
 \item[(c)] Formula (\ref{e3.2}) holds.  
  
  \end{enumerate}
 \end{lemma}
 
 \begin{lemma}
\label{L3.3}
Let $ F \in \cF^{[m;s,s']} $ be such that $ \theta_F = 0. $
\begin{enumerate}
\item[(a)] If $ m > 1, $ or $ m = 1 $ and $ (s,s') \notin \ZZ^2, $ then $ F = 0. $
\item[(b)] If $ m = 1 $ and $ s, s' \in \ZZ,  $ then $ F \in \CC \widehat{R}^A. $
\item[(c)] If $ m = \half, $ then $ F \in \CC \widehat{R}^B_{2s\bmod 2 \ZZ, \ 2s' \bmod 2 \ZZ} $.
\end{enumerate}
 \end{lemma}
 
 Now we give proofs of Theorems \ref{t2.2} and \ref{t2.3}, using these three lemmas. Proofs of the latter will be given in Section 4.
 
\begin{proof}[Proof of Theorem 2.2]
 Let $ F \in \cF^{[m;s,s']}, $ then $ \theta_F \in \cG^{[m;s,s']} $ by Lemma \ref{L3.1}. Hence by Lemma \ref{L3.2} (a), 
\[ \theta_F = c_+ \theta^{+[m,s]}\,\, (\hbox{resp.}\,\,  c_- \theta^{-[m,s]})\,\,\hbox{if}\,\,  s' \in \ZZ \,\, (\hbox{resp.}\,\,   \in \half + \ZZ ), 
\,\,\hbox{where}\,\, c_{\pm} \in \CC.
\]
 Hence, by Lemma \ref{L3.2} (b), (c),
 \[ \theta_{F +\frac{ c_{\pm}}{2 \sqrt{m}} \tilde{\varphi}^{\pm[m,s]}} = 0,
  \]
 \noindent and therefore, by Lemma \ref{L3.3}(a),
 \[ F + \frac{c_{\pm}}{2 \sqrt{m}} \tilde{\varphi}^{\pm[m,s]} = 0, \]
 \noindent proving Theorem \ref{t2.2}.
 \end{proof} 
 
 \begin{proof}[Proof of Theorem 2.3]
 Let $ F \in \cF^{[m;s,s']}. $ Then, by (\ref{e3.3})
and Lemma \ref{L3.3} (b) and (c) we have:
 \begin{equation}
 \label{e3.4}
 F + \frac{c_{\pm}}{2 \sqrt{m}} \tilde{\varphi}^{\pm [m,s]} \in \CC \widehat{R}^A \,\mbox{ if }\, m = 1\, \mbox{ and }\, s, s' \in \ZZ \ (\mbox{resp. }\, \in \CC \widehat{R}^B_{2s \bmod 2 \ZZ, \ 2s' \bmod 2 \ZZ}\, \mbox{ if }\, m = \half). 
 \end{equation}
 \noindent It remains to check that 
the functions $ \widehat{R}^A $ (resp. $ \widehat{R}^B_{ab} $)
, given by (\ref{e2.21}) (resp. (\ref{e2.22}) lie in $ \cF^{[1;0,0]} $ (resp. in $ \cF^{[\half;s,s']} $ if $ a \equiv 2s \bmod 2 \ZZ, b \equiv 2s' \bmod 2 \ZZ $). Property (F2) is obvious for $ \widehat{R}^A $ and it follows from Remark \ref{r2.5} for $ \widehat{R}^B. $ Properties (F3) and (F4) of these functions follow from the elliptic transformation properties of $ \vartheta_{ab} $,
$a,b=0$ or $1$, $n\in \ZZ$, which can be easily derived from  \cite{M}, page 19: 
\begin{equation}
\label{e3.5}
\vartheta_{ab} (\tau, z + \frac{n}{2}) = (-1)^{abn+\frac{an(1-n)}{2}} \vartheta_{a,b+n\!\!\!\!\mod \! 2\ZZ} (\tau, z), 
\end{equation}
\begin{equation}
\label{e3.6}
\vartheta_{ab} (\tau, z + \frac{n}{2} \tau) = (-i)^{bn} q^{- \frac{n^2}{8}} e^{-\pi i n z} \vartheta_{a+n\!\!\!\! \mod\! 2\ZZ,b} (\tau, z).
\end{equation}

Finally, the functions $ \widehat{R}^A $ and $  \widehat{R}^B_{ab} $ being mock theta functions, due to the denominator identities 
(cf. \cite{KW1},  \cite{G}, \cite{GK}),
lie in the kernel of $ D $, and trivially they lie in the kernels of $ \bar{D} $ and $ \Delta, $ hence satisfy (F5).
\end{proof}

\section{Proof of Lemmas \ref{L3.1} -- \ref{L3.3}}

First, we prove the following lemma, which gives one of the inclusions of Lemma \ref{L3.2}(a). Its proof is along the same lines as the proof of Proposition 13.3 in \cite{K2}.
\begin{lemma}
\label{L4.1}
Let $ m \in \half \zp $ and $ s \in \half \ZZ, $ let $ s' = 0 $ (resp. $ = \half $), which corresponds to $ + $ (resp.~$ - $) case below. 
\begin{enumerate}
\item[(a)] If $ g(\tau, u, v, t) $ is a function, satisfying conditions (G1)--(G3), then
\[ g(\tau, u, v, t) = e^{2 \pi i m t} \hspace{-2em} \sum_{n, n' \in s + \ZZ \bmod 2 m \ZZ} \hspace{-2em} c_{n, n'} (\tau) e^{\frac{\pi i s'}{m}(n + n')} \Theta^{\pm}_{n,m} (- \bar{\tau}, 2 \bar{v}) \Theta^{\pm}_{n', m} (\tau, 2u) \]
for some real analytic functions $c_{n,n'}(\tau)$ in $ \tau,\,\, \Im \tau > 0. $
\item[(b)] If, in addition, $ g(\tau, u, v, t)  $ satisfies condition (G5), then
\[ g(\tau, u, v, t) = e^{2 \pi i m t} \hspace{-2em} \sum_{n, n' \in s + \ZZ \bmod 2 m \ZZ} \hspace{-2em} c_{n, n'}  e^{\frac{\pi i s'}{m}(n + n')} \Theta^{\pm}_{n,m} (- \bar{\tau}, 2 \bar{v}) \Theta^{\pm}_{n', m} (\tau, 2u) \]
for some $ c_{n,n'}  \in \CC. $
\item[(c)] If $ g(\tau, u,v,t) $ satisfies all the conditons (G1)--(G5), then 
\[ g(\tau, u, v, t) = c \ \theta^{\pm [m,s]} (\tau, u, v, t), \]
where $ c \in \CC $ and $ \theta^{\pm[m,s]} $ is given by (\ref{e3.3}).
Consequently $ \cG^{[m;s,s']} \subset \CC \theta^{\pm [m,s]}$.  
\end{enumerate}
\end{lemma}
\begin{proof}
First we prove (a) in the case $ s' = 0. $ Since, by (G2), 
$ g(\tau, u, v, t) $ is a holomorphic function  in $u$, and it
satisfies (G1) and (G3) (i) for $ k =0, $ we have its Fourier series decomposition in $ u: $
\begin{equation}
\label{e4.01}
g \tuvt = e^{2 \pi i m t} \sum_{n \in \ZZ} A_n (\tau, v) e^{2 \pi i (n + s) u}. 
\end{equation}
\noindent Hence we have:
\begin{equation}
\label{e4.02}
g (\tau, u + j \tau, v ,t) = e^{2 \pi i m t} \sum_{n \in \ZZ} A_n (\tau, v) e^{2 \pi i (n+s) u} q^{(n+s) j}, 
\end{equation}
\begin{equation}
\label{e4.03}
\begin{split}
 e^{-4 \pi i mju} q^{-mj^2} g \tuvt &=  e^{2 \pi i m t} \sum_{n \in \ZZ} A_n (\tau, v) q^{-mj^2} e^{2 \pi i (n-2mj+s)u}, \\
&=  e^{2 \pi i m t} \sum_{n \in \ZZ} A_{n+2mj} (\tau, v) q^{-mj^2} e^{2 \pi i (n+s) u}.\\
\end{split}
\end{equation}
\noindent Due to (G3)(iii) for $ k=0, $ comparing the coefficients of $ e^{2\pi i (n+s) u} $ in the RHS of (\ref{e4.02}) and (\ref{e4.03}), we obtain:
\[ A_n (\tau, v) q^{(n+s)j} = A_{n+2mj}(\tau,v) q^{-mj^2}. \]
\noindent Hence we have:
\begin{equation}
\label{e4.04}
B_n = B_{n + 2mj}, \mbox{ where } B_n = A_n q^{-\frac{(n+s)^2}{4m}}.
\end{equation}
\noindent Substituting this in (\ref{e4.01}), we obtain:
\begin{equation}
\label{e4.05}
g \tuvt = e^{2\pi i m t} \hspace{-1em} \sum_{p \in \ZZ \bmod 2m \ZZ} \hspace{-1em} B_p (\tau, v) \Theta^+_{p+s, m} (\tau, 2u).
\end{equation}

Next, we compute the coefficients $ B_p (\tau, v), p \in \ZZ $. Since, 
by (G2),  $ B_p (\tau, v)  $ is holomorphic in $ \bar{v} $ , by (G3)(i) with $ j=0, $ we have its decomposition in the following Fourier series:
\begin{equation}
\label{e4.06}
B_p (\tau, v) = \sum_{n \in \ZZ} \al^{(p)}_n (\tau) e^{2 \pi i (n+s) \bar{v}}.
\end{equation}
\noindent Arguing as above and using (G3)(ii) for $ j = 0, $ we obtain
\begin{equation}
\label{e4.07}
c^{(p)}_n (\tau) = c^{(p)}_{n-2mk} (\tau), \mbox{ where } c^{(p)}_n (\tau) = \al^{(p)}_n (\tau) e^{\frac{\pi i \bar{\tau}}{2m}(n+s)^2}.
\end{equation}
\noindent Substituting this in (\ref{e4.06}), we obtain: 
\begin{equation}
\label{e4.08}
B_p (\tau, v) = \hspace{-1.5em} \sum_{p' \in \ZZ \bmod 2m \ZZ} \hspace{-1.5em} c^{(p)}_{p'}(\tau) \Theta^+_{p' +s, m} (- \bar{\tau}, 2 \bar{v}),
\end{equation}
\noindent which completes the proof in this case. The proof in case $ s' = \half $ is similar. 

In order to prove (b), we apply $ D $ to both sides of the formula in (a) to obtain, using (G5):
\[ 
\begin{split}
0  = & \ 8 \pi i m \ e^{2 \pi i m t} \hspace{-1em} \sum_{n,n' \in s + \ZZ \bmod 2m \ZZ} \hspace{-1em} \frac{\partial c_{n,n'} (\tau)}{\partial \tau } \Theta^\pm_{n,m} (- \bar{\tau}, 2 \bar{v}) \Theta^\pm_{n',m} (\tau, 2u)\\
& + e^{2 \pi i m t}  \hspace{-1em} \sum_{n,n' \in s + \ZZ \bmod 2m \ZZ} \hspace{-1em} c_{n,n'} (\tau) \Theta^\pm_{n,m} (- \bar{\tau}, 2 \bar{v}) \left( 8 \pi i m \frac{\partial }{\partial \tau } - \left( \frac{\partial }{\partial u}\right)^2 \right) \Theta^\pm_{n'm} (\tau, 2u). \\
\end{split}
\]
\noindent Since one has
\begin{equation}
\label{e4.09}
\left( 8 \pi i  m \ \frac{\partial}{\partial \tau} - \left( \frac{\partial}{\partial u}\right)^2\right) \, \Theta^\pm_{n,m} (\tau, 2u) = 0,
\end{equation}
\noindent we deduce that $ \frac{\partial c_{n,n'} (\tau)}{\partial \tau} =0. $

Similarly, applying $ 4 \frac{\partial }{\partial t} \frac{\partial }{\partial \bar{\tau}} + \left( \frac{\partial }{\partial \bar{v}}\right)^2,$ we deduce that $  \frac{\partial c_{n,n'} (\tau)}{\partial \bar{\tau}} =0, $ proving (b). 

In order to prove (c), introduce the following functions, for $ m \in \half \zp, \ j,k \in \half \ZZ: $
\[ h^\pm_{j,k} (\tau, u, v) = \Theta^\pm_{j,m} (- \bar{\tau}, 2 \bar{v}) \Theta^\pm_{-k,m} (\tau, 2u). \]
\noindent These functions have the following elliptic transformation properties, which follow from (\ref{e1.15}).
\begin{equation}
\label{e4.10}
\begin{split}
h^\pm_{j,k} \left(\tau, u + \frac{p}{2m}, v + \frac{p}{2m }\right)& = 
e^{\frac{\pi i(j-k)p}{m}} h^\pm_{j,k} (\tau, u, v), \ p \in \ZZ; \\
h^\pm_{j,k} \left(\tau, u + \frac{\tau}{2m}, v + \frac{\tau}{2m }\right)& = 
e^{\frac{\pi i }{2m} (\bar{\tau} - \tau)}
e^{2\pi i(\bar{v}-u)}
h^{\pm}_{j-1,k-1} (\tau, u, v). \\
\end{split}
\end{equation}
\noindent By (b) we have:
\[ g \tuvt = e^{2 \pi i m t} \hspace{-2em} \sum_{n,n' \in s + \ZZ \bmod 2m \ZZ} \hspace{-2em} c_{n,-n'} e^{\frac{\pi i s'}{m} (n-n')} h^\pm_{n,n'} (\tau, u, v), \]
\noindent where $ c_{n, -n'} \in \CC. $ Using (\ref{e4.10}), and condition (G4)(i), we obtain that $ c_{n, -n'} \neq 0 $ only if $ n-n' \in 2m \ZZ, $ hence
\[ g \tuvt = e^{2 \pi i m t} \hspace{-2em}\sum_{n \in s + \ZZ \bmod 2 m \ZZ} \hspace{-2em} c_n h^\pm_{n,n} (\tau, u, v).  \]
\noindent Using (\ref{e4.10}) and condition (G4) (ii), we deduce that all the $ c_n $'s are equal, proving (c).
\end{proof}

\begin{lemma}
\label{L4.2}
The functions $ \tilde{\varphi}^{\pm[m,s]} (\tau, u, v, t)$ satisfy the elliptic transformation properties (F3).
\end{lemma}
\begin{proof}
It follows from \cite{KW5}, Theorem 1.5(c).
\end{proof}
\begin{lemma}
\label{L4.3}
$ D \tilde{\varphi}^{\pm [m,s]} (\tau, u,v,t) = 0, $ where $ D $ is defined in (\ref{e2.19}).
\end{lemma}
\begin{proof}
\noindent Letting 
 \[ f_m (\tau, v)  =  \left(\frac{m}{y}\right)^{\half} e^{-4 \pi m y a^2}, \] 
 \noindent and using (\ref{e2.17}), we observe that
\begin{equation}
\label{e4.11}
\left( 8 \pi i m \frac{\partial }{\partial \tau} + \left(\frac{\partial}{\partial v}\right)^2 \right) f_m (\tau, v)=0.
\end{equation}
We claim that
\begin{equation}
\label{e4.12}
D \left( \frac{\partial}{\partial \bar{v}} \varphi^{\pm[m,s]}_{\mathrm{add}}\right) = 0,
\end{equation}
\noindent where $ \varphi^{\pm [m,s]}_{\mathrm{add}} (\tau, u, v, t) $ is obtained from  $ \Phi^{\pm [m,s]}_{\mathrm{add}} (\tau, z_1, z_2, t) $ given by (\ref{e2.10}), by the change of variables (\ref{e2.12}).

Indeed, by Lemma 1.7(1) from \cite{KW4} and using (\ref{e2.18}), we obtain:
\begin{equation}
\label{e4.13}
 \frac{\partial}{\partial \bar{v}} R^{\pm}_{j,m} (\tau, v) = 2i f_m (\tau, v) \Theta^{\pm}_{j,m} (- \bar{\tau}, 2 \bar{v}). 
\end{equation}
Hence
\[  \frac{\partial}{\partial \bar{v}} \varphi^{\pm[m,s]}_{\mathrm{add}} (\tau, u, v, t) = -i e^{2 \pi i m t} \hspace{-1em}\sum_{\substack{j \in s + \ZZ \\ s \leq j < s + 2m} } \hspace{-1em} f_m (\tau, v) \Theta^{\pm}_{j,m} (- \bar{\tau}, 2 \bar{v}) \Theta^{\pm}_{-j,m} (\tau, 2u).\]
\noindent Applying $ D $ to both sides of this equation, we obtain
\[ D \left(\frac{\partial}{\partial \bar{v}} \ \varphi^{\pm[m,s]}_{\mathrm{add}} \right) = \left( 8 \pi i m \ \frac{\partial}{\partial \tau} + \left( \frac{\partial}{\partial v}\right)^2 - \left(\frac{\partial}{\partial u}\right)^2\right) \frac{\partial}{\partial \bar{v}} \ \varphi^{\pm[m,s]}_{\mathrm{add}} = 0, \]
\noindent using (\ref{e4.09}) and (\ref{e4.11}). Hence (\ref{e4.12}) is proved. 

Next, $ D \varphi^{\pm[m,s]} = 0, $  since $ \varphi^{\pm[m,s]} $ is a mock theta function. Hence $ D \tilde{\varphi}^{\pm[m,s]} =   D \varphi_{\mathrm{add}}^{\pm[m,s]} $ is a real analytic function, since $ \varphi_{\mathrm{add}}^{\pm[m,s]} $ is. On the other hand,
\[ \frac{\partial}{\partial \bar{v}} D \tilde{\varphi}^{\pm[m,s]}= \frac{\partial}{\partial \bar{v}} D \varphi^{\pm[m,s]}_{\mathrm{add}}= D \frac{\partial}{\partial \bar{v}}  \varphi^{\pm[m,s]}_{\mathrm{add}}= 0\]
\noindent by (\ref{e4.12}). Hence, being real analytic, the function $ D \tilde{\varphi}^{\pm[m,s]} $ is holomorphic in $ v. $

Since $ D $ commutes with the transformations $ p_{\al} $ and $ t_{\al} $ (defined by (\ref{e1.06a}) and (\ref{e1.06})), it follows from Lemma \ref{L4.2} that $ D \tilde{\varphi}^{\pm[m,s]} $ satisfies (F3) as well. Hence, letting $ \tilde{\psi} = D \tilde{\varphi}^{\pm[m,s]} $ for short, we have, in particular:
\begin{equation}
\label{e4.14}
\tilde{\psi} (\tau, u, v+2, t) = \tilde{\psi} (\tau, u, v, t)\,\, \mbox{ and }\,\, \tilde{\psi} (\tau, u, v-2 \tau,t) = e^{8 \pi i m (\tau -v)} \tilde{\psi} (\tau, u, v, t). 
\end{equation}

\noindent Now, by (\ref{e3.5}) and (\ref{e3.6}), we have:
\begin{equation}
\label{e4.15}
\vartheta_{11} (\tau, v+2) = \vartheta_{11} (\tau, v), \quad \vartheta_{11} (\tau, v - 2\tau) = e^{4 \pi i (v-\tau)} \vartheta_{11} (\tau, v).
\end{equation}
\noindent Letting $ f(v) = \tilde{\psi} (\tau, u, v, t) \vartheta_{11} (\tau, v)^{2m}, $ we have by (\ref{e4.14}) and (\ref{e4.15}):
\[ f(v+2) = f(v), \quad f(v-2 \tau) = f(v), \]
\noindent and since $ \tilde{\psi} $ is holomorphic in $ v, f(v) $ is holomorphic in $ v $ as well. Thus $ f(v) $ is a bounded holomorphic function in $ v \in \CC,  $ hence $ f(v)  $ is constant in 
$ v $. Since $ \vartheta_{11} (\tau, 0) = 0, $ we deduce that $ f(v) = 0, $ hence $ \tilde{\psi} = 0,$ proving the lemma.
\end{proof}
\begin{lemma}

\label{L4.4}
Let $ f = f(\tau, u, v, t) $ be a real analytic function of degree $ m$ on the domain $X$.
Then
\begin{enumerate}
\item[(a)] The function $ \theta_f $ (defined by (\ref{e3.1})) is holomorphic in $ \bar{v} $ if and only if $ \Delta f = 0. $
\item[(b)] $ \left( 8 \pi i m \frac{\partial}{\partial \bar{\tau}} + \left( \frac{\partial }{\partial \bar{v}}\right)^2\right) \theta_f = \theta_{\bar{D} f}$.
\item[(c)] If $ \Delta f = 0, $ then $ D \theta_f = \theta_{Df} $ (see (\ref{e2.19}) for the definition of $ D, \bar{D}, $ and $ \Delta $).
\end{enumerate}
\end{lemma}

\begin{proof}
Let, for simplicity of notation, $ g = \frac{i}{2} \theta_f. $ Then, by (\ref{e3.1}),
\begin{equation}
\label{e4.16}
g = y^{\half} e^{4 \pi m a^2 y} \frac{\partial f}{\partial \bar{v}}.
\end{equation}

\noindent In order to perform the computations of partial derivatives, we will need the following formulas, which are immediate by (\ref{e2.17}):
\begin{equation}
\label{e4.17}
\frac{\partial a}{\partial v} = \frac{1}{2 i y}, \quad \frac{\partial a}{\partial \bar{v}} = -\frac{1}{2 i y},  \quad \frac{\partial a}{\partial \tau} = \frac{ia}{2 y}, \quad \frac{\partial y}{\partial \tau} = \frac{1}{2 i}, \quad \frac{\partial}{\partial \bar{\tau}} = D_{\bar{\tau}} - a \frac{\partial}{\partial \bar{v}}.
\end{equation}
\noindent Using the first of these formulas and (\ref{e4.16}), we obtain:
\[ \frac{\partial g}{\partial v} = -y^{\half} e^{4 \pi m a^2 y} \Delta f, \]
\noindent proving (a). 

Using (\ref{e4.16}) and (\ref{e4.17}), we obtain:
\begin{equation}
\label{e4.18}
\frac{\partial g}{\partial \bar{\tau}} = \frac{1}{4i} y^{-\half} e^{4 \pi m a^2 y} \frac{\partial f}{\partial \bar{v}} - 2 \pi i m y^{\half} a^2 e^{4 \pi m a^2 y} \frac{\partial f}{\partial \bar{v}} -y^{\half} a e^{4 \pi m a^2 y} \frac{\partial^2 f }{\partial \bar{v}^2} + y^{\half} e^{4 \pi m a^2 y} \frac{\partial }{\partial \bar{v}} D_{\bar{\tau}} f, 
\end{equation}
\begin{equation}
\label{e4.19}
\frac{\partial g}{\partial \tau} = \frac{1}{4i} y^{-\half} e^{4 \pi m a^2 y} \frac{\partial f}{\partial \bar{v}} + 2 \pi i m y^{\half} e^{4 \pi m a^2 y} \frac{\partial f}{\partial\bar{v}} + y^{\half} e^{4 \pi m a^2 y} \frac{\partial^2 f }{ \partial \bar{v} \partial \tau },
\end{equation}
 \begin{equation}
 \label{e4.20}
 \begin{split}
 \frac{\partial^2 g}{\partial \bar{v}^2}  = &- 2 \pi m y^{-\half} e^{4 \pi m a^2 y}\frac{\partial f}{ \partial \bar{v}} -(4 \pi m)^2 y^{\half} a^2 e^{4 \pi m a^2 y}\frac{\partial f}{ \partial \bar{v}}   + 8 \pi m i y^{\half} a e^{4 \pi m a^2 y} \frac{\partial^2 f }{ \partial \bar{v}^2} \\ & + y^{\half} e^{4 \pi m a^2 y} \frac{\partial^3 f}{ \partial \bar{v}^3 }. \\
\end{split} 
 \end{equation}
\noindent Multiplying (\ref{e4.18}) by $ 8 \pi i m  $ and adding to it (\ref{e4.20}), we obtain (b).

Similarly, we find
\[ \frac{\partial^2 g}{ \partial v^2}\! \! =\!\! - 2 \pi m y^{-\half} e^{4 \pi m a^2 y} \frac{\partial f}{ \partial \bar{v}} - (4 \pi m)^2 y^{\half} a^2 e^{4 \pi m a^2 y} \frac{\partial f}{ \partial \bar{v}} - 8 \pi i m y^{\half} a e^{4 \pi m a^2 y} \frac{\partial^2 f}{ \partial v \partial \bar{v} } + y^{\half} e^{4 \pi m a^2 y} \frac{\partial^3 f }{ \partial v^2 \partial \bar{v}}. \]
If $ \Delta f =0, $ we can replace in this formula $ \frac{\partial^2 f }{ \partial v \partial \bar{v} } $ by $ 4 \pi i m a \frac{\partial f}{ \partial \bar{v}} , $ hence we have
\begin{equation}
\label{e4.21}
\frac{\partial^2 g }{ \partial v^2} = -2 \pi m y^{-\half} e^{4 \pi m a^2 y} \frac{\partial f}{ \partial \bar{v}} + (4 \pi m)^2 y^{\half} a^2 e^{4 \pi ma^2 y} \frac{\partial f}{ \partial \bar{v}} + y^{\half} e^{4 \pi m a^2 y} \frac{\partial^3 f}{ \partial v^2 \partial \bar{v}}, 
\end{equation}
\noindent provided that $ \Delta f = 0. $ Multiplying (\ref{e4.19}) by $ 8 \pi im  $ and adding to it (\ref{e4.21}), we obtain:
\begin{equation}
\label{e4.22}
\left(8 \pi i m \frac{\partial }{ \partial \tau } + \frac{\partial^2 }{ \partial v^2}\right) g = y^{\half} e^{4 \pi m a^2 y} \frac{\partial }{ \partial \bar{v}} \left(8 \pi i m \frac{\partial }{ \partial \tau } + \frac{\partial^2 }{ \partial v^2}\right) f,
\end{equation}
\noindent provided that $ \Delta f = 0.  $ Subtracting from (\ref{e4.22}) the equation 
\[ \frac{\partial^2 g}{ \partial u^2} = y^{\half} e^{4 \pi m a^2 y} \frac{\partial }{ \partial \bar{v}} \frac{\partial^2 }{ \partial u^2 }f, \]
\noindent we obtain
\[ Dg = y^{\half} e^{4 \pi m a^2 y} \frac{\partial }{ \partial \bar{v}} Df, \]
\noindent provided that $ \Delta f =0, $ proving (c).
\end{proof}
\begin{lemma}
\label{L4.5}
Let $ m \in \half \zp, \ s \in \half \ZZ, $ and let $ f = f(\tau, u, v, t) $ 
be a function of degree $m$ on the domain $X$, which is real analytic 
in $\bar{v}$. Let $ j,k \in \ZZ. $ Then
\begin{enumerate}
\item[(a)] $ f(\tau, u + j, v + k, t) = e^{2 \pi i s (j+k)} f(\tau, u, v, t) $
\noindent implies that 

$  \theta_f (\tau, u + j, v+k, t) = e^{2 \pi i s (j+k)} \theta_f (\tau, u, v, t).  $
\item[(b)] $ f(\tau, u + j \tau, v + k \tau, t) = (\pm 1)^{j+k} e^{4 \pi i m 
(kv - ju)} q^{m(k^2 - j^2)} f(\tau, u, v, t) $ \noindent implies that 

$ \theta_f (\tau, u + j \tau, v + k \tau, t) 
= (\pm 1)^{j+k} e^{4 \pi i m (k \bar{v} - j u)} e^{2 \pi i m (k^2 \bar{\tau} - j^2 \tau)} \theta_f (\tau, u, v, t).   $
\item[(c)] $ f \left(\tau, u + \frac{k}{2m}, v + \frac{k}{2m}, t \right) = f(\tau, u, v, t)$ 
\noindent implies that 

$  \theta_f \left(\tau, u + \frac{k}{2m}, v + \frac{k}{2m}, t\right) = \theta_f (\tau, u, v, t).   $
\item[(d)] $ f \left(\tau, u + \frac{k \tau}{2m}, v + \frac{k \tau}{2m}, t \right)  = e^{2 \pi i k (v-u)} f(\tau, u, v, t)$ \noindent implies that 

$  \theta_f \left(\tau, u + \frac{k \tau}{2m}, v + \frac{k \tau}{2m}, t \right) = e^{2 \pi i k (\bar{v}-u)} e^{\frac{\pi i k^2}{2m}(\bar{\tau} - \tau)} \theta_f (\tau, u, v, t).  $
\end{enumerate}
\end{lemma}
\begin{proof}
(a) is obvious. In order to perform calculations, it is convenient to rewrite the definition (\ref{e3.1}) of $ \theta_f, $ using (\ref{e2.17}) and (\ref{e2.18}), as follows:
\begin{equation}
\label{e4.23}
\theta_f (\tau, u, v, t) = -(2i)^{\half} (\tau - \bar{\tau})^{\half} e^{-2 \pi i m \frac{(v- \bar{v})^2}{\tau - \bar{\tau}}} \frac{\partial f}{\partial \bar{v}}.
\end{equation}
\noindent It is clear that it suffices to prove (b) for $ j = 0. $ By the assumption on $ f, $ we have, using (\ref{e4.23}):

\[ 
\begin{aligned}
& \theta_f (\tau, u,  v + k \tau, t)  \\
& =  -(2i)^{\half} (\tau - \bar{\tau})^{\half} (\pm 1)^k e^{-2 \pi i m \left(\frac{(v- \bar{v})^2}{\tau - \bar{\tau}} + 2k (v - \bar{v}) + k^2 (\tau - \bar{\tau}) \right)} e^{4 \pi i m k v} e^{2 \pi i m k^2 \tau} \frac{\partial f}{\partial \bar{v}} (\tau, u, v, t) \\
& =  (\pm 1)^k e^{4 \pi i m k \bar{v}} e^{2 \pi i m k^2 \bar{\tau}} \theta_f (\tau, u, v, t),
\end{aligned} \]
\noindent proving (b). The proof of (c) and (d) is similar.
\end{proof}

\begin{lemma}
\label{L4.6}
\[ \bar{D} R^{\pm}_{j,m} (\tau, v) = 0. \]
\end{lemma}
\begin{proof}
First, we compute $ D_{\bar{\tau}} R^{\pm}_{j,m} (\tau, v), $ using that $ D_{\bar{\tau}} y = \frac{i}{2}: $
\[ D_{\bar{\tau}} R^{\pm}_{j,m} (\tau, v) = - \frac{i y^{-\half}}{2 \sqrt{m}} e^{-4 \pi m a^2 y} \hspace{-1em}\sum_{n \equiv j \bmod 2 m \ZZ} \hspace{-1em} (\pm 1)^{\frac{n-j}{2m}} (n -2ma) e^{-\frac{\pi i n^2 \bar{\tau}}{2m}} e^{2 \pi i n \bar{v}}. \]
\noindent Letting $ n = j + 2mk, \ k \in \ZZ,$ in this equation, we obtain:
\[ 
\begin{aligned}
 D_{\bar{\tau}} R^{\pm}_{j,m} (\tau, v) & = \frac{-i y^{-\half}}{2 \sqrt{m}} e^{-4 \pi m a^2 y} \left( \sum_{k \in \ZZ} (\pm 1)^k (j + 2mk) e^{-2 \pi i \bar{\tau} m \left( k + \frac{j}{2m} \right)^2}  e^{2 \pi i m \left( k + \frac{j}{2m} \right) 2 \bar{v}}  \right. \\
& \hspace{1em} \left. - 2 ma \sum_{k \in \ZZ} (\pm 1)^k e^{-2 \pi i \bar{\tau} m \left( k + \frac{j}{2m}\right)^2}  e^{2 \pi i m \left( k + \frac{j}{2m}\right) 2 \bar{v}} \right) \\
& = - \frac{1}{4 \pi \sqrt{m}} y^{-\half} e^{-4 \pi m a^2 y} \left( \frac{\partial}{\partial \bar{v}} \Theta^{\pm}_{j,m} (-\bar{\tau}, 2 \bar{v}) - 4 \pi i m a \ \Theta^{\pm}_{j,m} (-\bar{\tau}, 2 \bar{v}) \right)\\
& = - \frac{1}{4 \pi \sqrt{m}}  \left( y^{-\half} e^{-4 \pi m a^2 y} \frac{\partial }{\partial \bar{v}}  \Theta^{\pm}_{j,m} (- \bar{\tau}, 2 \bar{v}) - 4 \pi i m a y^{- \half} e^{-4 \pi m a^2 y}  \Theta^{\pm}_{j,m} (- \bar{\tau}, 2 \bar{v}) \right). \\ 
\end{aligned} \]
\noindent Using that $ \frac{\partial a}{\partial \bar{v}} = \frac{i}{2y}, $ we deduce that 
\[ D_{\bar{\tau}} R^{\pm}_{j,m} (\tau, v)  = - \frac{1}{4 \pi \sqrt{m}} \ \frac{\partial }{\partial \bar{v}}  \left( y^{-\half} e^{-4 \pi m a^2 y} \Theta^{\pm}_{j,m} (- \bar{\tau}, 2 \bar{v})\right). \]
\noindent Hence, using formula (\ref{e4.13}), we obtain
\[ D_{\bar{\tau}} R^{\pm}_{j,m} (\tau, v)  = \frac{-1}{8 \pi i m } \ \frac{\partial^2 }{\partial \bar{v}^2} R^{\pm}_{j,m} (\tau, v), \]
\noindent proving the lemma. 
\end{proof}
\begin{lemma}
\label{L4.7}
Let $ m \in \half \zp, \ s \in \half \ZZ, $ and let $ a \in \QQ $ be such that $ am \in \ZZ$. Then
\begin{enumerate}
\item[(a)] $ R^{\pm}_{j,m} (\tau, v + \frac{a}{2}) = e^{\pi i j a} R^{\pm}_{j,m} (\tau, v). $
\item[(b)] $ \varphi^{\pm[m,s]}_{\mathrm{add}} (\tau, u + \frac{a}{2}, v + \frac{a}{2}, t)  = \varphi^{\pm[m,s]}_{\mathrm{add}} (\tau, u, v, t).$
\end{enumerate}
\end{lemma}
\begin{proof}
(a) is Lemma 1.6 from \cite{KW4}. (b) follows from the proof of Theorem 1.11(2) from \cite{KW4}.
\end{proof}

\begin{proof}[\textbf{Proof of Lemma \ref{L3.1}.}] Let $ F \in \cF^{[m;s,s']}. $ Then property (F1) of $ F $ obviously implies property (G1) of $ \theta_F. $ By Lemma \ref{L4.4}a, properties (F2) and the third one of (F5) imply property (G2) of $ \theta_F.  $ By Lemma \ref{L4.5}, properties (F3) and (F4) of $ F $ imply properties (G3) and (G4) of $ \theta_F. $ Finally, by Lemma \ref{L4.4} (b), (c), the first two properties in (F5) of $ F $ imply property (G5) of $ \theta_F. $
\end{proof}

\begin{proof}[\textbf{Proof of Lemma \ref{L3.2}.}]
By Lemma \ref{L4.1}(c), in order to prove (a), it suffices to show that $ \theta^{+[m,s]} $ (resp. $ \theta^{-[m,s]} $) is contained in $ \cG^{[m;s,s']} $ if $ s' \in \ZZ $ (resp. $ \in \half + \ZZ $). This holds by claim (c). 
In order to prove claim (c), note that, since
 $ \frac{\partial \varphi^{\pm[m,s]}}{\partial\bar{v}} =0, $ we have: 
\[\begin{aligned} \theta_{\tilde{\varphi}^{\pm[m,s]}} & = -2i y^{\half} e^{4 \pi  m  a^2 y} \frac{\partial \varphi^{\pm [m,s]}_{\mathrm{add}}}{\partial \bar{v}} \\& = iy^{\half} e^{4 \pi  ma^2 y} e^{2 \pi i m t} \hspace{-1em}\sum_{ \substack{ k \in s + \ZZ \\ s \leq k < s + 2m}} \hspace{-0.5em} \frac{\partial}{\partial \bar{v}} R^{\pm}_{k,m} (\tau, v) \Theta^{\pm}_{-k,m} (\tau, 2u).
\end{aligned} \] 
\noindent Substituting (\ref{e4.13}) in this formula completes the proof.  

It remains to prove claim (b). For this we use Lemma \ref{L4.2}, Lemma \ref{L4.3}, and it remains to check (F4) and 
(F5).

It is straightforward to check that the function $ \varphi^{\pm[m,s]} $ satisfies property (F4)(i). It follows by Lemma \ref{L4.7} that the function $ \tilde{\varphi}^{\pm[m,s]} $ satisfies this property too. Then, using $ S $-invariance of the function $ \tilde{\varphi}^{\pm[m,s]},  $ given by Corollary 1.6  from \cite{KW5}, property (F4)(ii) of $  \tilde{\varphi}^{\pm[m,s]} $ holds as well.

The function $  \tilde{\varphi}^{\pm[m,s]} $ is annihilated by $ D, \, \bar{D} $ and $ \Delta $ by Lemmas \ref{L4.3}, \ref{L4.6}, and \ref{L4.4}(a) respectively, which proves that it satisfies property (F5).
\end{proof}

\begin{proof}[\textbf{Proof of Lemma \ref{L3.3}(a).}] Let 
\begin{equation}
\label{e4.24}
 \tilde{\vartheta} (\tau, u, v) = \vartheta_{11} (\tau, v - u) \vartheta_{11} (\tau, v +u), \quad F_1 (\tau, u, v, t) = \tilde{\vartheta} (\tau, u, v) F(\tau, u, v, t).  
\end{equation}
\noindent Since $ \nabla F = 0, $ we have $ \frac{\partial}{\partial \bar{v}} F = 0, $ and hence $ \frac{\partial}{\partial \bar{v}} F_1 = 0. $ Due to condition (F2) on $ F $ we deduce that the function $ F_1 (\tau, u, v, t)  $ is holomorphic in $ v. $ By conditions (F3) on $ F $ we obtain
\[ F (\tau, u, v + 2, t) = F(\tau, u, v, t), \quad F(\tau, u, v-2 \tau, t) = e^{8 \pi i m (\tau - v)} F(\tau, u, v, t). \]
\noindent From this, and (\ref{e3.5}) and (\ref{e3.6}) we obtain:
\begin{equation}
\label{e4.25}
F_1(\tau, u, v+2)= F_1 (\tau, u, v), \quad F_1 (\tau, u, v-2 \tau) = 
e^{8 \pi i (m-1) (\tau -v)} F_1 (\tau, u, v).
\end{equation}
Letting  $ h(v) = \vartheta_{11} (\tau, v)^{2m-2} F_1 (\tau, u, v, t),  $ we deduce from (\ref{e3.5}), (\ref{e3.6}) and (\ref{e4.25}):
\[ h(v+2) = h(v), \quad h(v - 2 \tau) = h(v). \]
Thus, the function $ h(v) $ is a bounded holomorphic function if $ m \geq 1, $ and $ h(0) = 0 $ if $ m > 1 $ (since $ \vartheta_{11} (\tau, 0) = 0 $). This completes the proof in the case $m>1$.

In the case $m=1$ we obtain that $F_1(\tau,u,v,t)$ is independent of $v$.
Letting $j=0$ and $k=1$ in (F3)i (resp in (F3)ii), and using (\ref{e3.5})
and (\ref{e3.6}), we obtain
\begin{equation}
\label{e4.26}
F_1(\tau,u,v+1,t)=e^{2\pi is}F_1(\tau,u,v,t)\, \hbox{and} 
F_1(\tau,u,v+\tau,t)=e^{2\pi is'}F_1(\tau,u,v,t). 
\end{equation}
Hence $F_1=0$ if either
$s$ or $s'$ is not an integer.  
\end{proof}

In order to prove Lemma \ref{L3.3}(b) and (c), we need the following lemma. We omit its proof, which is along the same lines as the proof of Proposition 13.3 in \cite{K2}, and Lemma \ref{L4.1}.
\begin{lemma}
\label{L4.8}
Let $ m \in \half \zp, \ j \in \half \ZZ $, and let $ \al $ be a non-zero real number. Let $ f (\tau, z) $ be a function, which is real analytic in $ \tau, \ \Im \tau > 0, $ and holomorphic in $ z \in \CC. $ Suppose  that 
\[ f \left(\tau, z + \frac{2}{\al} \tau \right) = \pm q^{-m} e^{-2 \pi i m \al z} f(\tau, z)\,\,\hbox{and}\,\, f \left(\tau, z + \frac{2}{\al} \right) = e^{2 \pi i j } f(\tau, z). \]
 Then there exist real analytic functions $ c_n (\tau), \ n \in j + \ZZ, \ 0 \leq n < 2m, $ such that
\[ f(\tau, z) = \sum_{\substack{n \in j + \ZZ \\ 0 \leq n < 2m}} c_n(\tau) \Theta^{\pm}_{n,m} (\tau, \al z). \]
\qed
\end{lemma}

\begin{proof}[\textbf{Proof of Lemma \ref{L3.3}(b).}]
By (\ref{e4.26}), the function $ F_1 = \vartheta_{11} (\tau, v -u) \vartheta_{11} (\tau, v + u) F(\tau, u, v, t)  $ is independent on $ v,  $ and by the condition on $ F, $ it is independent on $ \bar{v}. $ Hence the function $ f (\tau, u): = e^{-2 \pi i t} F_1 (\tau, u, v, t) $ is a function, which is real analytic in $ \tau, \Im \tau > 0, $ and holomorphic in $ u \in \CC. $ This function satisfies all conditions of Lemma \ref{L4.8} with $ m = \half, j = \half, \al = 4,  $ due to conditions (F3) on $ F $ and (\ref{e3.5}), (\ref{e3.6}). Hence we have (cf. Example \ref{x1.09}):
\[ f (\tau, u) = c_{\half} (\tau) \, \Theta^-_{\half, \half} (\tau, 4 u, 0) = -i  c_{\half} (\tau) \vartheta_{11} (\tau, 2u), \]
\noindent where $ c_{\half} (\tau) $ is a real analytic function in $ \tau, \Im \tau > 0. $

Thus, by (\ref{e2.21}) we have:
\[ F(\tau, u, v, t) =  \frac{c_{\half} (\tau)}{i\,\eta (\tau)^3} \widehat{R}^A (\tau, u, v, t). \]
\noindent Since $ F $ is annihilated by $ D $ and $ \bar{D} $ by (F5), we obtain that $ \frac{c_{\half} (\tau)}{\eta (\tau)^3} $ is annihilated by $ \frac{\partial }{\partial \tau } $ and $ \frac{\partial }{\partial \bar{\tau}} $ respectively. Hence $ \frac{c_{\half} (\tau)}{\eta (\tau)^3} $ is a constant, completing the proof.
\end{proof}

In order to prove Lemma \ref{L3.3}(c), we need the following lemma.
\begin{lemma}
\label{L4.9}
Let $ f(\tau,u,v,t) = \widehat{R}^B_{ab} (\tau, u, v, t) g (\tau, u)$, where $ g(\tau, u) $ is a real analytic function in $ \tau, \Im \tau > 0, $ and a meromorphic function in $ u \in \CC. $ Suppose that $ Df = 0 $ and $ \bar{D}f = 0.  $ Then $ g $ is a constant.
\end{lemma}
\begin{proof}
Recall that $ D \widehat{R}^B_{ab} = 0 $ (see the end of the proof of Theorem \ref{t2.3}). Hence, we have, by the assumption on $ f: $
\[ 0 = Df = \widehat{R}^B_{ab} \left( 4 \pi i \frac{\partial g}{\partial \tau} - \frac{\partial^2 g }{\partial u^2} \right) - 2 \frac{\partial \widehat{R}^B_{ab}}{\partial u} \ \frac{\partial g}{\partial u}.\]
Dividing both sides by $ \widehat{R}^B_{ab}, $ we obtain:
\begin{equation}
\label{e4.27}
\left( 4 \pi i \frac{\partial g}{\partial \tau} - \frac{\partial^2 g }{\partial u^2}\right) - 2 \frac{\partial \log \widehat{R}^B_{ab}}{\partial u} \ \frac{\partial g}{\partial u} = 0.
\end{equation}
\noindent Since $ g $ is independent of $ v, $ applying $ \frac{\partial }{\partial v} $ to both sides of (\ref{e4.27}), we obtain:
\[ -2 \left( \frac{\partial }{\partial v} \ \left( \frac{\partial \log \widehat{R}^B_{ab}}{\partial u}\right)\right) \frac{\partial g}{\partial u} = 0.\]
\noindent It follows that $ \frac{\partial g}{\partial u} = 0.$ Hence, by (\ref{e4.27}), we obtain that $ \frac{\partial g}{\partial \tau} = 0. $

Next, we have 
\[ 0 = \bar{D}f = \left(4 \pi i \frac{\partial }{\partial \bar{\tau}} + \left( \frac{\partial }{\partial \bar{v}} 
\right)^2 \right) f = 4 \pi i \frac{\partial f}{\partial \bar{\tau}} ,\]
\noindent since $ f $ is meromorphic in $ u. $ Hence
\[ \frac{\partial }{\partial \bar{\tau}} (\widehat{R}^B_{ab}\ g ) = \widehat{R}^B_{ab} \frac{\partial g }{\partial \bar{\tau}} = 0, \]
\noindent and $ \frac{\partial g}{\partial \bar{\tau}} = 0, $ proving the lemma.
\end{proof}

\begin{proof}[\textbf{Proof of Lemma \ref{L3.3}(c).}]
Letting $ j = 0, \ k = 1 $ in (F3), we have
\[ F(\tau, u, v+1, t) = e^{2 \pi i s} F(\tau,u,v,t), \quad F(\tau, u, v + \tau, t) = e^{2 \pi i s'} q^{\half} e^{2 \pi i v} F(\tau, u, v, t). \] 
\noindent Hence, by (\ref{e3.5})  and (\ref{e3.6}), the function $ F_1, $ defined by (\ref{e4.24}), satisfies
\begin{equation}
\label{e4.28}
F_1 (\tau, u, v+1, t ) = e^{2 \pi i s} F_1 (\tau, u, v, t), \quad F_1 (\tau, u, v + \tau, t) = e^{2 \pi i s'} q^{-\half} e^{-2 \pi i v} F_1 (\tau, u, v, t).
\end{equation}
Since $ \cF^{[m;s,s']} $ depends only on $ s, s' \bmod \ZZ, $ we may assume that $ s, s' $ are equal to 0 or $ \half. $ Applying Lemma \ref{L4.8} for $ m = \half, \ j = s, \ \al = 2, $ we deduce from (\ref{e4.28}):
\begin{equation}
\label{e4.29}
F_1 (\tau, u, v, t) = e^{\pi i t} c^\pm_s (\tau, u) \ \Theta^\pm_{s, \half} (\tau, 2v),
\end{equation}
\noindent where + (resp. -) corresponds to $ s' = 0 $ (resp. $ = \half $), and $ c^\pm_s (\tau, u) $ is a function, real analytic in $ \tau, \Im \tau > 0, $ and holomorphic in $ u \in \CC.  $

But we have (cf. Remark \ref{r1.10}):
\[ \Theta^\pm_{s, \half} (\tau, 2v) = (-i)^{4ss'} \vartheta_{2s, 2s'} (\tau, v). \]
\noindent Substituting this in (\ref{e4.29}), and dividing both sides of (\ref{e4.29}) by 
$\vartheta (\tau, v- u) \vartheta (\tau, v+ u)$,  
we obtain:
\begin{equation}
\label{e4.30}
F(\tau, u, v, t) = \widehat{R}^B_{2s, 2s'} (\tau, u, v, t) \ A(\tau, u),
\end{equation}
\noindent where $ A (\tau, u) $ is a real analytic function in $ \tau, \Im \tau > 0, $ and meromorphic in $ u $, and $ \widehat{R}^B_{ab} $ is given by (\ref{e2.22}).

Since, by condition (F5), $ DF = 0 $ and $ \bar{D}F = 0, $ we conclude, by Lemma \ref{L4.9}, that $ A(\tau, u) $ is a constant. This completes the proof. 
\end{proof}


\begin{thebibliography}{9999999}


\bibitem[Ap]{Ap} M. P. Appell, Sur le fonctions doublement periodique de troisieme espece, \textit{Annals Sci. l'Ecole norm. Sup., 3e serie} \textbf{1}, p. 135, \textbf{2}, p. 9, \textbf{3}, p. 9 (1884-1886).


\bibitem[G]{G} M. Gorelik, Weyl denominator identity for affine Lie 
superalgebras with non-zero dual Coxeter number, Japan. J. Algebra 
\textbf{337} (2011), 50-62. 

\bibitem[GK]{GK} M. Gorelik, V. G. Kac, Characters of (relatively) integrable 
modules over affine Lie superalgebras, 
Jpn. J. Math \textbf{10} (2015), No. 2, 135-235. arXiv:1406.6860 . 

\bibitem[K1]{K1} V. G. Kac, Lie superalgebras, \textit{Adv. Math.} \textbf{26}, No. 1 (1977), 8--96.

\bibitem[K2]{K2} V. G. Kac, Infinite-dimensional Lie algebras, Third edition, Cambridge University press, 1990. 

\bibitem[KP]{KP} V. G. Kac, D. H. Peterson, Infinite-dimensional Lie algebras, theta functions and modular forms, \textit{Adv. Math.} \textbf{53} (1984), 125--264. 

\bibitem[KW1]{KW1} V. G. Kac, M. Wakimoto, Integrable highest weight modules over affine superalgebras and number theory, \textit{Progress in Math.} \textbf{123}, Birkh\"{a}user, 1994, pp 415--456.

\bibitem[KW2]{KW2} V. G. Kac, M. Wakimoto, Integrable highest weight modules over affine superalgebras and Appell's function, \textit{Commun. Math. Phys.} \textbf{215} (2001), 631--682.

\bibitem[KW3]{KW3} V. G. Kac, M.Wakimoto, Representations of affine superalgebras and mock theta functions, \textit{Transf. Groups} \textbf{19} (2014), 387-455.

\bibitem[KW4]{KW4} V. G. Kac, M. Wakimoto, Representations of affine superalgebras and mock theta functions II, \textit{Adv. Math.} arXiv:1402.0727.

\bibitem[KW5]{KW5} V. G. Kac, M. Wakimoto, Representations of affine superalgebras and mock theta functions III, arXiv:1505.01047.

\bibitem[M]{M} D. Mumford, Tata lectures on theta I, Progress in Math. \textbf{28}, Birkhauser, 1983.

\bibitem[Za]{Za} D. Zagier, Ramanujan's mock theta functions and their applications (after Zwegers and Ono-Bringmann). Seminaire Bourbaki. Vol 2007-2008. Asterisque No. 326 (2009), Exp No. 986, vii-viii, (2010), 143-164 

\bibitem[Z]{Z} S. P. Zwegers, Mock theta functions, arXiv:0807.4834


\end{thebibliography}
\end{document}